\documentclass[11pt,a4paper]{amsart}
\usepackage[utf8x]{inputenc}
\usepackage[T1]{fontenc}
\usepackage{textcomp}
\usepackage{babel}
\usepackage{lmodern}
\usepackage{amssymb}
\usepackage[normalem]{ulem}

\usepackage{hyperref}
\usepackage{xcolor}
\hypersetup{%
    colorlinks=false,
    pdfborderstyle={/S/U/W 1}
}

\usepackage{xspace}
\usepackage{enumitem}
\setenumerate[1]{label=(\roman{*}), ref=(\roman{*})}
\setenumerate[2]{label=(\alph{*}), ref=(\alph{*})}

\theoremstyle{plain}
\newtheorem{thm}{Theorem}[section]
\newtheorem*{thm*}{Theorem}
\newtheorem*{mainthm*}{Main theorem}
\newtheorem*{lem*}{Lemma}
\newtheorem{cor}[thm]{Corollary}
\newtheorem{prop}[thm]{Proposition}
\newtheorem{lem}[thm]{Lemma}

\theoremstyle{definition}
\newtheorem{defn}[thm]{Definition}%[section]
\newtheorem{exmp}[thm]{Example}

\theoremstyle{remark}
\newenvironment{claim}[1]{\vskip 1mm\par\noindent\underline{Claim\space#1:}}{}
\newenvironment{claimproof}[1]{\vskip
  1mm\par\noindent\underline{Proof:}\space#1}{\hfill $\blacksquare$
  \vskip 1mm}

\DeclareMathOperator{\Lip}{Lip}
\DeclareMathOperator{\lip}{lip}

\begin{document}
\title{Duality of Lipschitz-free spaces over ultrametric spaces}

% \date{\today}
\author[T.~A.~Abrahamsen]{Trond A.~Abrahamsen}
\address[T.~A.~Abrahamsen]{Department of Mathematics, University of
  Agder, Post\-boks 422, 4604 Kristiansand, Norway.}
\email{trond.a.abrahamsen@uia.no}
\urladdr{http://home.uia.no/trondaa/index.php3}

\author[V.~Lima]{Vegard Lima}
\address[V.~Lima]{Department of Mathematics, University of
  Agder, Postboks 422, 4604 Kristiansand, Norway.}
\email{vegard.lima@uia.no}

\author[A.~Ostrak]{Andre Ostrak}
\address[A.~Ostrak]{Department of Mathematics, University of
  Agder, Postboks 422, 4604 Kristiansand, Norway.}
\email{andre.ostrak@uia.no}

\begin{abstract}
  We give a metric characterisation of when the Lipschitz-free space over
  a separable ultrametric space is a dual Banach space.
  In the case where the Lipschitz-free space has a predual,
  we show that this predual is $M$-embedded if and only if
  the metric space is proper.
  We show that for ultrametric spaces, the little Lipschitz space is
  always an $M$-ideal in the corresponding space of Lipschitz
  functions, and we show that this is not the case for metric spaces in
  general, thus answering a question posed by Werner in the negative.
  Finally, we show that the space of Lipschitz functions of an
  ultrametric space contains a strongly extreme point.
\end{abstract}


\maketitle

\section{Introduction}

Let $M$ be a pointed metric space, i.e. a metric space with a fixed point $0$.
We denote by $\Lip_0(M)$ the Banach space of all Lipschitz functions $f\colon M\to \mathbb{R}$ satisfying $f(0)=0$, with the norm
\begin{equation*}
  \|f\|=\sup\left\{\frac{|f(x)-f(y)|}{d(x,y)}\colon x,y\in M,\ x\neq
    y\right\}.
\end{equation*}
Denote by $\delta\colon M\to \Lip_0(M)^*$ the canonical isometric
embedding of $M$ into $\Lip_0(M)^*$, given by
$\delta\mapsto \delta_x$, where $\delta_x(f)=f(x)$ for $f\in \Lip_0(M)$.
The \emph{Lipschitz-free space} over $M$, denoted by $\mathcal{F}(M)$,
is the closed linear span of $\delta(M)$ in $\Lip_0(M)^*$.
It is known that $\mathcal{F}(M)^*=\Lip_0(M)$.

The study of when the Lipschitz-free space $\mathcal{F}(M)$ 
is a dual space has garnered a lot of attention 
over the years by several authors
\cite{MR4745886, MR4402669, MR3376824, MR3778926, MR3693571, MR2068975, MR1386169, MR3792558}. 
However, outside of a few examples, 
this line of study has mostly been confined to proper metric spaces $M$. 
An omnibus theorem given in~\cite[Theorem~3.2]{MR4402669} 
shows that for a proper metric space $M$, 
the Lipschitz-free space $\mathcal{F}(M)$ is a dual 
if and only if $M$ is purely $1$-unrectifiable. 
One of the few results extending outside the proper spaces
was given by Weaver in~\cite{MR1386169},
who showed that the Lipschitz-free space $\mathcal{F}(M)$
of certain rigidly locally compact spaces $M$ is a dual.
The first aim of this paper is to break the
confines of proper metric spaces and
characterise the duality of $\mathcal{F}(M)$ 
for all separable ultrametric spaces $M$.

The duality and related properties of Lipschitz-free spaces over
ultrametric spaces have previously been studied in \cite{MR3376824},
where it was shown that the Lipschitz-free space over an ultrametric
space has the metric approximation property, and that $\mathcal{F}(M)$ is a dual space if $M$
is a proper ultrametric space.
We borrow the machinery for the proof of our main theorem
from \cite{MR3530906}, where it was shown that
if $M$ is a separable ultrametric space, then the space
$\mathcal{F}(M)$ has a monotone Schauder basis and is isomorphic to
$\ell_1$.

An ultrametric space $M$ is said to be spherically complete (see,
e.g. ~\cite{MR2810332}) if any nested sequence of balls in $M$ has
nonempty intersection. Our main result is the following.

\begin{mainthm*}
  Let $M$ be a complete separable ultrametric space.
  The following are equivalent:
  \begin{enumerate}
  \item\label{main:i}
    $\mathcal{F}(M)$ is a dual Banach space;
  \item\label{main:ii}
    $\mathcal{F}(M)$ is $1$-complemented in its bidual
    $\mathcal{F}(M)^{**}$;
  \item\label{main:iii}
    $M$ is spherically complete.
  \end{enumerate}
\end{mainthm*}

\noindent 
More generally, we show that the implications
\ref{main:i}$\Rightarrow$\ref{main:ii}$\Rightarrow$\ref{main:iii}
hold even if $M$ is not separable.
Section~\ref{sec:non-dual-lipschitz} is dedicated
to the proof of \ref{main:ii}$\Rightarrow$\ref{main:iii}.
The rest of the proof of the Main theorem comes
in the first half of Section~\ref{sec:dual-lipschitz-free},
where we construct a predual $Y$ for $\mathcal{F}(M)$
when $M$ is separable and spherically complete.
In Section~\ref{sec:dual-lipschitz-free} we also show that
our predual $Y$ coincides with the predual given
by Dalet \cite{MR3376824} when $M$ is a proper
ultrametric space.

The second half of Section~\ref{sec:dual-lipschitz-free},
we study $M$-ideals.
We denote by $\lip^u_0(M)$ the Banach space of all uniformly locally
flat functions in $\Lip_0(M)$ that are also flat at infinity.
See the Preliminaries section below for definitions.
It is known that if $M$ is any proper metric space, then for any $\varepsilon > 0$,
$\lip_0^u(M)$ is $(1+\varepsilon)$-isometric to a subspace of $c_0$.
This was shown by Kalton \cite[Theorem~6.6]{MR2068975}
for compact metric spaces and extended to proper metric spaces
by Dalet \cite[Lemma~3.9]{MR3376824}.
It follows that $\lip_0^u(M)$ is an $M$-ideal in $\Lip_0(M)$
when $M$ is a proper metric space such that
$\lip_0^u(M)$ is a predual of $\mathcal{F}(M)$.
We show that for any ultrametric space $M$, $\lip_0^u(M)$
is an $M$-ideal in $\Lip_0(M)$.
For spherically complete $M$, our predual $Y$ of 
$\mathcal{F}(M)$ is an $M$-ideal in its bidual $\Lip_0(M)$
if and only if $M$ is proper.

In~\cite{MR4472717}, Werner asked whether for all compact
metric spaces $M$, the space $\lip_0^u(M)$ is an $M$-ideal of
$\Lip_0(M)$.
In Example~\ref{exmp:noWerner}, we answer this question negatively by giving an example of
a (non-ultrametric) compact metric space for which this is not the case.

Finally, in Section~\ref{sec:strongly-extr-points},
we show that for an ultrametric space $M$,
the space $\Lip_0(M)$ has a strongly extreme point and is therefore
not locally almost square.

\section{Preliminaries}
\label{sec:preliminaries}

We study real Banach spaces only.
Given a real Banach space $X$, we denote the closed unit ball, the
unit sphere, and the dual space of $X$ by $B_X$, $S_X$, and $X^*$,
respectively.
Recall that if $X$ is a Banach space and $E$ is a subspace of $X^*$,
then $E$ is \emph{$a$-norming} where $a \ge 1$ if
\begin{equation*}
  \|x\| \le a \sup_{x^* \in B_E} |x^*(x)|
  \qquad \text{for all $x \in X$}.
\end{equation*}
If $a = 1$, we call $E$ \emph{norming}.
That $E$ is $a$-norming is equivalent to requiring that $B_{X^*}$
is contained in the weak$^*$-closure of $a B_E$.

Let $Y$ be a subspace of a Banach space $X$
with annihilator $Y^\perp$ in $X^*$.
Recall that $Y$ is an \emph{$M$-ideal} in $X$
if there is a projection $P$ on $X^*$
with $\ker P = Y^\perp$ such that for all $x^* \in X^*$,
\begin{equation*}
  \|x^*\| = \|Px^*\| + \|x^* - Px^*\|.
\end{equation*}
A Banach space $X$ is \emph{$M$-embedded} if
$X$ is an $M$-ideal in its bidual $X^{**}$.

Throughout the paper, $M$ will stand for a metric space
with metric $d$.
We will assume without mention that $M$ is pointed, that is,
there is a distinguished point $0 \in M$.
Given a metric space $M$, we denote
by $B(x,r)$ the \emph{open} ball in $M$ centered at $x$
with radius $r$, i.e.
$B(x,r) = \{ y \in M : d(x,y) < r \}$.
The corresponding closed ball is
$B[x,r] = \{ y \in M : d(x,y) \le r \}$.
Recall that a metric space $M$ is \emph{proper}
if every closed ball is compact.

For $x,y\in M$ with $x\neq y$, we define a \emph{molecule} $m_{x,y}$
as the norm one element
$\tfrac{\delta_x-\delta_y}{d(x,y)}$ in $\mathcal{F}(M)$.

\subsection{Preliminaries on little Lipschitz spaces.}
Unfortunately, there is no uniform agreement in the literature on
what it means for a Lipschitz function to be
flat at infinity or what the space of little Lipschitz functions
should be, so let us start with some definitions.

\begin{defn}
  Let $M$ be a metric space and let $f \in \Lip_0(M)$.
  Then
  \begin{enumerate}
  \item\label{d:flat_lf}
    $f$ is \emph{locally flat} if for every $p \in M$
    and every $\varepsilon > 0$, there exists $\delta > 0$
    such that
    \begin{equation*}
      x,y \in B(p,\delta)
      \implies
      |f(x) - f(y)| < \varepsilon d(x,y);
    \end{equation*}
  \item\label{d:flat_ulf}
    $f$ is \emph{uniformly locally flat} if for every
    $\varepsilon > 0$, there exists $\delta > 0$ such that
    \begin{equation*}
      d(x,y) < \delta
      \implies
      |f(x) - f(y)| < \varepsilon d(x, y);
    \end{equation*}
  \item\label{d:flat_fai}
    $f$ is \emph{flat at infinity} if
    \begin{equation*}
      \lim_{r \to \infty}
      \sup_{\overset{x\;\text{or}\;y \notin B[0,r]}{x \neq y}}
      \frac{f(x) - f(y)}{d(x,y)} = 0.
    \end{equation*}
  \end{enumerate}
\end{defn}
Note that $f \in \Lip_0(M)$ is flat at infinity
if and only if
\begin{equation*}
  \lim_{r \to \infty} \|f|_{M \setminus B[0,r]}\| = 0.
\end{equation*}
The proof of the non-trivial implication
is similar to the one in
Lemma~4.18 in Weaver's book \cite{MR3792558}.
Note, however, that Weaver's definition of flat at infinity
is slightly different from ours in that Weaver requires 
$f$ to be flat outside of a compact set.
If $M$ is proper (aka. boundedly compact), then
the two definitions agree.
The above widely used definition of flat at infinity
is more suitable for our purposes when dealing with
non-proper metric spaces.

For a metric space $M$, we define
\begin{equation*}
  \lip_0^u(M) := \{ f \in \Lip_0(M) :
  f\ \text{is uniformly locally flat and flat at infinity}
  \}.
\end{equation*}
If $M$ is proper, $\lip_0^u(M)$ coincides with the space
$\lip_0(M)$ as defined in \cite[Definition~3.1]{MR4402669}
or \cite[Definition~4.15]{MR3792558}.
Since our metric spaces will not all be proper,
we use the notation $\lip_0^u(M)$ to highlight
that we assume the little Lipschitz functions
to be \emph{uniformly} locally flat.

\subsection{Preliminaries on ultrametric spaces.}
Recall that a metric space $M$ with metric $d$ is called
\emph{ultrametric} if it satisfies the strong triangle inequality,
i.e. for all $x,y,z\in M$,
\begin{equation*}
  d(x,y)\leq \max\{d(x,z),d(z,y)\}.
\end{equation*}

Let us list some results about ultrametric spaces,
which we will use throughout the paper.
The proof of the following lemma is straightforward,
see e.g. \cite[Proposition~1.3.28]{MR3931701}.

\begin{lem}\label{lem: prop_ultram}
  Let $M$ be an ultrametric space, let $x,y,z \in M$, and let $r>0$.
  \begin{enumerate}
  \item\label{item:1}
    If $d(x,z) \neq d(y,z)$,
    then $d(x,y) = \max\{d(x,z),d(y,z)\}$;
  % \item\label{item:2}
  %   If $y\in B(x,r)$ and $z\not\in B(x,r)$, then $d(y,z)\geq r$;
  \item\label{item:3}
    If $d(x,y)<r$, then $B(x,r) = B(y,r)$;
  \item\label{item:4}
    $B(x,r)$ is both open and closed.
  \end{enumerate}
\end{lem}

We define a sequence of partitions for an ultrametric space.
\begin{lem}\label{lem: partitioning}
    Let $M$ be an ultrametric space and let $q\in (0,1)$. There exist subsets $C_n$ of $M$, $n\in \mathbb{Z}$, satisfying for all $n\in \mathbb{Z}$,
    \begin{enumerate}[label=(\alph*)]
        \item\label{item:a} $0\in C_n$;
        \item\label{item:b} $C_n\subset C_{n+1}$;
        \item\label{item:c} $d(x,y)\geq q^n$ for all $x,y\in C_n$ with $x\neq y$;
        \item\label{item:d} $M = \bigcup_{x\in C_n}B(x,q^n)$.
    \end{enumerate}
\end{lem}

\begin{proof}
    Fix $n\in \mathbb{Z}$. 
    First, we see that there exists a subset $C_n$ of $M$ 
    satisfying \ref{item:a}, \ref{item:c}, and \ref{item:d}. 
    Define
    \begin{align*}
    \mathcal{A} = \{&B\subset M\colon \text{$0\in B$; for all $x,y\in B$ with $x\neq y$, $d(x,y)\geq q^n$}\}.
    \end{align*}
    A standard application of Zorn's lemma gives us
    a maximal (by inclusion) set $C_n$ in $\mathcal{A}$. 
    Assume for contradiction that $M \neq \bigcup_{x\in C_n}B(x,q^n)$
    and let $z\in M\setminus \bigcup_{x\in C_n} B(x,q^n)$. 
    For all $x\in C_n$, we have $d(x,z)\geq q^n$, and therefore, $C_n\cup\{z\}$ is in $\mathcal{A}$, 
    which contradicts $C_n$ being maximal.
        
    Define $C_{n+1}$ by first taking for any $x\in C_n$ a maximal set $A_x$ from 
    \begin{align*}
    \mathcal{A}_{x} = \{&B\subset M\colon \text{$x\in B$ and} \\
    &\text{for all $y_1,y_2\in B$ with $y_1\neq y_2$, $q^{n+1}\leq d(y_1,y_2)< q^n$}\},
    \end{align*}
    and note that $\bigcup_{y\in A_x} B(y,q^{n+1}) = B(x,q^n)$. 
    Indeed, if there existed $z\in B(x,q^n)\setminus \bigcup_{y\in A_x} B(y,q^{n+1})$, 
    then for all $y\in A_x$,
    \[
    q^{n+1}\leq d(y,z)\leq \max\{d(y,x),d(x,z)\}< \max\{q^n, q^n\} = q^n, 
    \]
    which would contradict $A_x$ being maximal because $A_x\cup \{z\}\in \mathcal{A}_x$. Let $C_{n+1} = \bigcup_{x\in C_n} A_{x}$ and note that $C_{n+1}$ satisfies \ref{item:a}, \ref{item:c}, and \ref{item:d} when replacing $n$ with $n+1$, and that $C_n\subset C_{n+1}$. Continue this process iteratively to define $C_{n+k}$ for all $k\in \mathbb{N}$.
    
    We now define $C_{n-1}$. First, define an equivalence relation on $C_n$ with $x\equiv y$ if and only if $d(x,y)< q^{n-1}$. It is clear that $\equiv$ is reflexive and symmetric; for transitivity, we note that if $d(x,y)<q^{n-1}$ and $d(y,z)<q^{n-1}$, then $d(x,z)<q^{n-1}$. Pick an element from each equivalence class and let $A$ be the collection of such elements. To make sure that $0\in C_{n-1}$, take $x\in A$ with $x\equiv 0$, let $C_{n-1}=(A\setminus \{x\})\cup\{0\}$, and note that $C_{n-1}$ satisfies \ref{item:a}, \ref{item:c}, and \ref{item:d} when replacing $n$ with $n-1$, and that $C_{n-1}\subset C_{n}$. Continue this process iteratively to define $C_{n-k}$ for all $k\in \mathbb{N}$.
    \end{proof}

\subsection{Preliminaries on spherical completeness.}
    The notion of spherical completeness is maybe not well-known,
    so we include a few results and definitions that we will need.
    Proofs we have left out can be found in, e.g. \cite{MR361697,MR2810332}.
    \begin{defn}
        A sequence of sets $(A_n)$ is said to be
          \emph{nested} if $A_{n+1} \subset A_n$ for all $n\in \mathbb{N}$.
    \end{defn}

    \begin{defn}
        An ultrametric space $M$ is said to be \emph{spherically complete} if every nested sequence of balls in $M$ has nonempty intersection.
    \end{defn}

    \begin{prop}
        A spherically complete ultrametric space is complete.
    \end{prop}

    \begin{defn}
        A sequence $(x_n)$ in $M$ is said to be \emph{pseudo-Cauchy} if, given $n_1,n_2,n_3\in \mathbb{N}$, $n_1<n_2<n_3$, we have $d(x_{n_3},x_{n_2})<d(x_{n_2},x_{n_1})$.
    \end{defn}

We note that a sequence $(x_n)$ in an ultrametric space
is Cauchy if and only if $\lim_n d(x_n,x_{n+1}) = 0$
(see e.g. \cite[Lemma~1.5]{MR2810332}).

    \begin{defn}
        An element $x\in M$ is said to be a \emph{pseudo-limit} of a pseudo-Cauchy sequence $(x_n)$ if, given $n_1,n_2\in \mathbb{N}$, $n_1<n_2$, we have $d(x_{n_2},x)<d(x_{n_1},x)$.
    \end{defn}

    \begin{prop}
      \label{prop:SpComplete=PseudoCauchyLimit}
        An ultrametric space is spherically complete if and only if every pseudo-Cauchy sequence in $M$ has a pseudo-limit.
    \end{prop}

    \begin{lem}\label{lem: not spherically complete}
    Let $M$ be a complete ultrametric space that is not spherically
    complete. Then there exists a nested sequence of balls
    \begin{equation*}
      B(x_1,r_1)\supsetneq B(x_2,r_2)\supsetneq \dotsb
    \end{equation*}
    in $M$ with empty intersection,
    where $(r_n)$ is a strictly decreasing sequence satisfying $\lim_n r_n>0$.
    \end{lem}
    
      \begin{proof}
        Let $(B(y_n, s_n))_{n\in \mathbb{N}}$ be a nested sequence of balls
        in $M$ with empty intersection. Since the intersection is empty
        there exists a subsequence $(n_k)$ of $\mathbb N$ such that for all
        $k\in \mathbb{N}$, 
        \[
          B(y_{n_k},s_{n_k})\neq B(y_{n_{k+1}},s_{n_{k+1}}).
        \]
        Clearly $\bigcap_{k} B(y_{n_k},s_{n_k}) \subset
        \bigcap_{n} B(y_{n},s_{n})=\emptyset.$

        For $k \in \mathbb N$, let $x_k = y_{n_k}$ and $r_k = s_{n_k}$. Then
        \[
          B(x_1,r_1)
          \supsetneq B(x_2,r_2)
          \supsetneq
          \dotsb
          \supsetneq B(x_k,r_k)
          \supsetneq\dotsb
        \]
        Observe that if
        $x \in B(x_{k}, r_{k}) \setminus B(x_{k+1}, r_{k+1}),$ then
        \[
        r_{k+1} \le d(x, x_{k+1}) \le \max\{d(x, x_k), d(x_k, x_{k+1})\}
        < r_k,
        \]
        so $(r_k)$ is strictly decreasing. Finally, if $\lim_k r_k=0$,
        then $(x_k)$ is a Cauchy sequence and 
        $\lim_k x_k\in \bigcap_k B(x_k,r_k)$. Therefore, $\lim_k r_k>0$.
      \end{proof}
      
\section{Non-dual Lipschitz-free spaces over ultrametric spaces}
\label{sec:non-dual-lipschitz}
In this section, we prove that the implication \ref{item:ii}$\Rightarrow$\ref{item:iii} of the Main theorem holds. We get an even stronger result, as we do not assume that $M$ is separable.

Let $\mathcal{U}$ be a free ultrafilter on $\mathbb{N}$.
Given a Banach space $X$, $r > 0$, and a sequence
$(x_n)$ in $r B_X$,
we have $w^*\textrm{-}\lim_{\mathcal{U}} x_n \in r B_{X^{**}}$, where
\[
w^*\textrm{-}\lim_{\mathcal{U}} x_n (f) = \lim_{\mathcal{U}} f(x_n)\qquad \text{for all $f\in X^*$.}
\]

\begin{prop}\label{prop: ultrametric not dual}
    Let $M$ be a complete ultrametric space. If $\mathcal{F}(M)$ is $1$-complemented in its bidual, then $M$ is spherically complete.
\end{prop}

\begin{proof}    
    Let 
    $T\colon \mathcal{F}(M)^{**}\to \mathcal{F}(M)$ be a projection onto with $\|T\|\leq 1$. 
    Assume for contradiction that $M$ is not spherically complete. By Lemma~\ref{lem: not spherically complete}, 
    there exists a nested sequence of balls 
    $B(x_1,r_1)\supsetneq B(x_2,r_2)\supsetneq \dotsb$ in $M$ with empty intersection,
    where $(r_n)$ is a strictly decreasing sequence satisfying $\alpha := \lim_n r_n>0$.
    We may and do assume that $0\not\in B(x_1,r_1)$.

    Let $\mathcal{U}$ be a free ultrafilter on $\mathbb{N}$ and
    define 
    \[
        F=w^*\textrm{-}\lim_{\mathcal{U}}\delta_{x_n}\in \mathcal{F}(M)^{**}\quad \text{and} \quad \phi = T(F)\in \mathcal{F}(M).
    \]
    For every $n\in \mathbb{N}$, we have $\|\delta_{x_n}-\phi\|< r_n$. 
    Indeed, for all $m>n$, we have
    \begin{align*}
    \|\delta_{x_n}-\delta_{x_m}\|&=d(x_n,x_m) < r_{n},
    \end{align*}
    and therefore, 
    \[
    \|\delta_{x_n}-\phi\|\leq \|T\| \|\delta_{x_n}-F\|\leq \|\delta_{x_n}-F\|< r_n.
    \]
    Write $\phi = \sum_{i=1}^\infty \lambda_i m_{p_i,q_i}$, where
    $\lambda_i>0$, $p_i,q_i\in M$, $i\in \mathbb{N}$, and
    $\sum_i\lambda_i<\infty$. For $n\in \mathbb{N}$, let
    \begin{align*}
        A_n &= B(x_n,r_{n})\setminus B(x_{n+1},r_{n+1})\subset M,\\
        I_n &=\{i\in \mathbb{N}\colon p_i\in A_n\},\\
        J_n &= \{i\in \mathbb{N}\colon q_i\in A_n\},
    \end{align*}
    and let
    \begin{align*}
        \beta_n &=\sum_{i\in I_n} \frac{\lambda_i}{d(p_i,q_i)}-\sum_{i\in J_n} \frac{\lambda_i}{d(p_i,q_i)}=\sum_{i\in I_n\setminus J_n} \frac{\lambda_i}{d(p_i,q_i)}-\sum_{i\in J_n\setminus I_n} \frac{\lambda_i}{d(p_i,q_i)}.
    \end{align*}

  \begin{claim}{1}
  Let $n\in \mathbb{N}$. If $x\in A_n$ and $y\in M\setminus A_n$, then $d(x,y)\geq r_{n+1}$.
  \end{claim}
  \begin{claimproof}{}
    We use Lemma~\ref{lem: prop_ultram} %\ref{item:2} and 
    \ref{item:3}.    
    If $y\in B(x_{n+1},r_{n+1})$, then  
    $ B(x_{n+1},r_{n+1})=B(y,r_{n+1})$, so $x\not\in B(y,r_{n+1})$.
    If $y\not\in B(x_n,r_n)=B(x,r_n)$, then $d(x,y)\geq r_n>r_{n+1}$.
  \end{claimproof}
  \begin{claim}{2}
     The series $\sum_{n\in \mathbb{N}}\beta_n$ converges absolutely.
  \end{claim}
  \begin{claimproof}
    Let $n\in \mathbb{N}$.
    If $i\in I_n\setminus J_n$,
    then $p_i\in A_n$ and $q_i\in M\setminus A_n$,
    which implies $d(p_i,q_i)\geq r_{n+1}>\alpha$ by Claim~1.
    Analogously, for all $i\in J_n\setminus I_n$,
    we have $d(p_i,q_i)>\alpha$.
    Therefore,
    \begin{align*}
      |\beta_n|
      &\leq
      \sum_{i\in I_n\setminus J_n} \frac{\lambda_i}{d(p_i,q_i)}
      + \sum_{i\in J_n\setminus I_n} \frac{\lambda_i}{d(p_i,q_i)} \\
      &\leq
      \sum_{i\in I_n\setminus J_n}\frac{\lambda_i}{\alpha}
      +
      \sum_{i\in J_n\setminus I_n}\frac{\lambda_i}{\alpha}
      \leq
      \sum_{i \in I_n}\frac{\lambda_i}{\alpha}
      +
      \sum_{i\in J_n} \frac{\lambda_i}{\alpha}.
    \end{align*}
    Given $\varepsilon > 0$, there exists $N\in \mathbb{N}$ so
    that $\sum_{i > N} \lambda_i < \frac{\varepsilon \alpha}{2}$.
    Note that for any $k\in \mathbb{N}$,
    \begin{equation*}
      \bigcup_{n\geq k}A_n \subset
      \bigcup_{n\geq k} B(x_n,r_{n})
      = B(x_k,r_k).
    \end{equation*}
    Since $\bigcap_{k\in \mathbb{N}} B(x_k,r_{k})=\emptyset$,
    we have $\limsup_k A_k=\emptyset$.
    Therefore, there exists $k\in \mathbb{N}$
    such that, for $i = 1,\ldots, N$, 
    $p_i,q_i\not\in \bigcup_{n\geq k}A_n$,
    i.e. $i \not\in \bigcup_{n\geq k} I_n$
    and $i \notin  \bigcup_{n\geq k} J_n$.
    Since $I_n \cap I_m = \emptyset = J_n \cap J_m$
    for all $n \neq m$,
    we get
    \begin{equation*}
      \sum_{n\geq k} |\beta_n|
      \leq
      \sum_{n\geq k}
      \sum_{i \in I_n}\frac{\lambda_i}{\alpha}
      +
      \sum_{n\geq k}
      \sum_{i\in J_n} \frac{\lambda_i}{\alpha}
      \le
      \frac{2}{\alpha} \sum_{i > N} \lambda_i
      < \varepsilon,
    \end{equation*}
    which proves our claim.
  \end{claimproof}
  Let $K=\{k\in \mathbb{N}\colon \beta_k>0\}$ and for $n\in \mathbb{N}$, let $K_n=K\cap \{1,\dotsc,n-1\}$.
  We show that $\sum_{k\in K}\beta_k\geq 1$. Let $n\in \mathbb{N}$, define
    \[
    f_n(x)=\begin{cases}
        r_1-r_{k+1},  & \text{if $x\in A_k$, $k\in K_n$;}\\
        r_1, & \text{if $x\in B(x_1,r_1)\setminus \big(\bigcup_{k\in K_n}A_k\big) $;}\\
        0, & \text{otherwise,}
    \end{cases}
    \]
    and use Claim~1 to verify that $f_n\in B_{\Lip_0(M)}$.
    Note that
    \begin{align*}
        r_n&\geq \|\delta_{x_n}-\phi\|\geq f_n(\delta_{x_n}-\phi)\\
        &= r_1-\sum_{k\in K_n}\big(\sum_{i\in I_k} \lambda_i \frac{r_1-r_{k+1}}{d(p_i,q_i)}-\sum_{i\in J_k} \lambda_i \frac{r_1-r_{k+1}}{d(p_i,q_i)}\big)\\
        &\qquad-\sum_{k\in \mathbb{N}\setminus K_n}\big(\sum_{i\in I_k} \lambda_i \frac{r_1}{d(p_i,q_i)}-\sum_{i\in J_k} \lambda_i \frac{r_1}{d(p_i,q_i)}\big)\\
        &= r_1-\sum_{k\in K_n}(r_1-r_{k+1})\beta_k -\sum_{k \in \mathbb{N}\setminus K_n} r_1\beta_k\\
        &\geq r_1-\sum_{k\in K_n}(r_1-r_{k+1})\beta_k -\sum_{k \in K\setminus K_n} r_1\beta_k.
    \end{align*}
    Taking $n\to \infty$ and using Claim~2, we get
    \begin{align*}
        \alpha &\geq r_1-\sum_{k\in K}(r_1-r_{k+1})\beta_k,
    \end{align*}
    which gives us
    \begin{align*}
    \sum_{k\in K}(r_1-r_{k+1})\beta_k \geq r_1-\alpha.
    \end{align*}
    Since $r_{k+1}> \alpha$ and $\beta_k>0$ for all $k\in K$, we get
    \begin{align*}
    (r_1-\alpha)\sum_{k\in K}\beta_k \geq  \sum_{k\in K}(r_1-r_{k+1})\beta_k \geq  r_1-\alpha.
    \end{align*}
    Therefore, $\sum_{k\in K}\beta_k \geq 1$.

    Let $n\in \mathbb{N}$, define
    \[
    g_n(x)=\begin{cases}
        r_{k+1},  & \text{if $x\in A_k$, $k\in K_n$;}\\
        0, & \text{otherwise,}
    \end{cases}
    \]
    and use Claim~1 to verify that $g_n\in B_{\Lip_0(M)}$. Note that
    \begin{align*}
        r_n&\geq \|\phi-\delta_{x_n}\|\geq g_n(\phi-\delta_{x_n})= \sum_{k\in K_n}\big(\sum_{i\in I_k} \lambda_i \frac{r_{k+1}}{d(p_i,q_i)}-\sum_{i\in J_k} \lambda_i \frac{r_{k+1}}{d(p_i,q_i)}\big)\\
        &=\sum_{k\in K_n} r_{k+1}\beta_k.
    \end{align*}
    Taking $n\to \infty$ and using Claim~2, we get $\alpha \geq \sum_{k\in K}r_{k+1}\beta_k$. However, we have $\alpha<r_{k+1}$ for all $k\in \mathbb{N}$ and $\sum_{k\in K} \beta_k\geq 1$. This is a contradiction.
\end{proof}

We give an example of a separable ultrametric space $M$ for which $\mathcal{F}(M)$ is not a dual space. It is notable that the metric space $M$ is bounded and uniformly discrete because to our knowledge, only one example of such a metric space $M$, for which it is known that $\mathcal{F}(M)$ is not a dual, has previously appeared in literature~\cite[Example~5.8]{MR3778926}.
\begin{exmp}
    Let $M=(\mathbb{N}\cup\{0\},d)$ be an ultrametric space, where the distances between two elements are defined for $m,n\in \mathbb{N}\cup\{0\}$, $m\neq n$,
    \[
    d(m,n)= 1+\frac{1}{2^{\min\{m,n\}}}.
    \]
    Note that $M$ is complete and that any subsequence of $\mathbb{N}$
    is a pseudo-Cauchy sequence in this space that has no pseudo-limit. Therefore, $M$ is not spherically complete and $\mathcal{F}(M)$ is not a dual space by Proposition~\ref{prop: ultrametric not dual}.
\end{exmp} 

The metric space $M$ in the previous example can be made spherically complete with the addition of a single element.
\begin{exmp}
    Let $M=(\mathbb{N}\cup\{0\}\cup\{\omega\},d)$ be an ultrametric space, where the distances between two elements are defined for $m,n\in \mathbb{N}\cup\{0\}\cup\{\omega\}$, $m\neq n$,
    \[
    d(m,n)=\begin{cases}
        1+\frac{1}{2^{n}}, & \text{if $m=\omega$};\\
        1+\frac{1}{2^{m}}, & \text{if $n=\omega$};\\
        1+\frac{1}{2^{\min\{m,n\}}}, & \text{otherwise}.
    \end{cases}
    \]
    Then $M$ is spherically complete because every pseudo-Cauchy sequence in $M$ is a subsequence of $\mathbb{N}\cup\{0\}$ and $\omega$ is the pseudo-limit for all such sequences.
\end{exmp}

The next example shows that we cannot apply Proposition~\ref{prop: ultrametric not dual} unless $M$ is complete.

\begin{exmp}\label{exmp: F(M) with dual where d(x_n,x_n+1) goes to 0}
    Let $M=(\mathbb{N}\cup\{0\}\cup\{\omega\},d)$ be an ultrametric space, where the distances between two elements are defined for $m,n\in \mathbb{N}\cup\{0\}\cup\{\omega\}$, $m\neq n$,
    \[
    d(m,n)=\begin{cases}
        \frac{1}{2^{n}}, & \text{if $m=\omega$;}\\
        \frac{1}{2^{m}}, & \text{if $n=\omega$;}\\
        \frac{1}{2^{\min\{m,n\}}}, & \text{otherwise}.
    \end{cases}
    \]
    Then $\mathcal{F}(M)$ is a dual space because $M$ is a compact
    ultrametric space \cite[Theorem~3.7]{MR3376824}. Take
    $N=M\setminus\{\omega\}$ and note that
    $\mathcal{F}(M)=\mathcal{F}(N)$ because $M$ is the closure of $N$.
    However, $N$ is not complete, and therefore, it is not spherically complete.
\end{exmp}

\section{Dual Lipschitz-free spaces over separable ultrametric spaces}
\label{sec:dual-lipschitz-free}
We construct and study the predual of $\mathcal{F}(M)$ 
for a separable spherically 
complete ultrametric space $M$.
Recall the Main theorem
stated in the introduction:

\begin{thm}\label{thm: main, duality characterisation}
  Let $M$ be a complete separable ultrametric space.
  The following are equivalent:
  \begin{enumerate}
  \item\label{item:i}
    $\mathcal{F}(M)$ is a dual Banach space;
  \item\label{item:ii}
    $\mathcal{F}(M)$ is $1$-complemented in its bidual
    $\mathcal{F}(M)^{**}$;
  \item\label{item:iii}
    $M$ is spherically complete.
  \end{enumerate}
\end{thm}

\begin{proof}
Any dual Banach space is $1$-complemented in its bidual,
so \ref{item:i}$\Rightarrow$\ref{item:ii}
holds even if we replace the Lipschitz-free space
$\mathcal{F}(M)$ with any Banach space. The implication
\ref{item:ii}$\Rightarrow$\ref{item:iii} holds by
Proposition~\ref{prop: ultrametric not dual}.

It remains to prove \ref{item:iii} $\Rightarrow$ \ref{item:i}.
We will show that if $M$ is spherically complete,
then the subset $Y$ of $\Lip_0(M)$,
given in Definition~\ref{defn: predual Y of F(M)} below,
satisfies  $Y^*=\mathcal{F}(M)$.
By \cite[Theorem~4]{MR336299}
(see also \cite[Theorem~IV.2]{MR886417}), 
it is enough to show that $Y$ is a closed linear subspace of
$\Lip_0(M)$ that is norming and that all $f\in Y$ attain their norm.
This will be done in Propositions~\ref{prop:Y-closed_sep_subspace_of_lip0u}~\ref{item:Y-props1} and~\ref{prop: norm-attainment}.
\end{proof}

Before proceeding, let us point out that Dalet's \cite{MR3376824}
result that for proper ultrametric spaces $M$, the space
$\mathcal{F}(M)$ is dual space, is immediate from the
above theorem.

\begin{cor}\label{cor:proper_is_dual}
  If $M$ is a proper ultrametric space,
  then $\mathcal{F}(M)$ is a dual Banach space.
\end{cor}

\begin{proof}
  Balls (both closed and open) in an ultrametric space are closed.
  If $M$ is proper, then balls are compact and
  the intersection of a nested sequence of balls has
  non-empty intersection, so $M$ is spherically complete.
\end{proof}

Let $M$ be a separable ultrametric space.
We recall and fix some notation from~\cite[Proof of Theorem~1]{MR3530906}, 
which we use throughout the section. Let $S=\{s_n\colon n\in \mathbb{N}\}$ 
be a countable dense one-to-one sequence in $M$ with $s_1=0$. 
For all $n\in \mathbb{N}$ and $x\in M$, denote 
 \begin{align*}
 S_n&=\{s_1,\dotsc,s_n\},\\
 I_n(x)&=\{k\leq n\colon d(x,S_n)=d(x,s_k)\},\\
 i_n(x)&=\min I_n(x).
 \end{align*}
 Define $r_n\colon M\to S_n$ by $r_n(x)=s_{i_n(x)}$.
 \begin{lem}[{\cite[Proof of Theorem~1]{MR3530906}}]\label{lem: r_n is a retract}
    Let $n\in \mathbb{N}$. The function $r_n\colon M\to S_n$ is a $1$-Lipschitz retraction, i.e. $r_n(x)=x$ for all $x\in S_n$, and for all $x,y\in M$,
 \[
 d(r_n(x),r_n(y)) \leq d(x,y).
 \]
 Furthermore, $r_n \circ r_m = r_{\min\{n,m\}}$.
\end{lem}

Let us introduce a subset of $\Lip_0(M)$, 
which will be our candidate for a predual of $\mathcal{F}(M)$ when $M$ is spherically complete.
\begin{defn}\label{defn: predual Y of F(M)}
  Let $Y$ be the set of Lipschitz functions $f\in \Lip_0(M)$
  satisfying:
  for all $\varepsilon>0$,
  there exists $N\in \mathbb{N}$ such that
  \begin{equation*}
    r_N(x) = r_N(y)\implies 
    |f(x)-f(y)| \le \varepsilon d(x,y).
  \end{equation*}
  We will sometimes write $N(\varepsilon)$ or $N_f(\varepsilon)$
  to highlight $N$'s dependence on $\varepsilon$ and $f$.
\end{defn}
\noindent
It is important to keep in mind that $Y$ is not necessarily uniquely defined by $M$;
the definition may depend on the ordering of a countable dense one-to-one sequence $S$ in $M$.
From now on, we assume that such an $S$ is fixed whenever $Y$ is mentioned.

\begin{prop}\label{prop:Y-closed_sep_subspace_of_lip0u}
  The set $Y$ has the following properties:
  \begin{enumerate}
  \item\label{item:Y-props1}
    $Y$ is a closed linear subspace of $\Lip_0(M)$
    that is norming;
  \item\label{item:Y-props2}
    $Y$ is a subspace of $\lip_0^u(M)$.
  \end{enumerate}
\end{prop}

\begin{proof}
  \ref{item:Y-props1}.
  Let us first show that $Y$ is a linear subspace.
  It is clear that $\alpha f \in Y$ if $f \in Y$
  and $\alpha \in \mathbb{R}$ (use $N(\varepsilon/|\alpha|)$).
  Assume $f,g \in Y$ and $\varepsilon > 0$.
  Find $N_f(\varepsilon/2)$ and $N_g(\varepsilon/2)$.
  Let $N = \max(N_f,N_g)$.
  If $r_N(x) = r_N(y)$, then using Lemma~\ref{lem: r_n is a retract}
  \begin{equation*}
    r_{N_f}(x) = r_{N_f}(r_N(x))
    = r_{N_f}(r_N(y)) = r_{N_f}(y)
  \end{equation*}
  and similarly for $N_g$ so that
  \begin{equation*}
    |(f+g)(x) - (f+g)(y)|
    \le
    |f(x) - f(y)| + |g(x) - g(y)|
    \le
    (\varepsilon/2 + \varepsilon/2)d(x,y).
  \end{equation*}
  This shows that $f + g \in Y$.

  Next we show that $Y$ is closed.  Assume that
  $(f_k)$ in $Y$ satisfies $f_k \to f \in \Lip_0(M)$.  
  Choose $n_0\in \mathbb{N}$ so
  that $\|f_{n_0} - f\| \le \varepsilon/2$ and choose $N\in \mathbb{N}$ so that
  $|f_{n_0}(x) - f_{n_0}(y)| \le \varepsilon d(x,y)/2$ whenever
  $r_N(x) = r_N(y)$.
  If $x, y \in M$ with $r_N(x) = r_N(y)$, then
  \begin{equation*}
    |f(x) - f(y)| \le
    |(f-f_{n_0})(x) - (f-f_{n_0})(y)|
    +
    |f_{n_0}(x) - f_{n_0}(y)|
    \le
    \varepsilon d(x,y)
  \end{equation*}
  as desired.

  Note that $\mathcal{F}(M)$ is isometrically isomorphic to $\mathcal{F}(S)$.
  To show that $Y$ is norming, it is enough, by \cite[Proposition~3.3]{MR2068975},
  to let $A$ be a finite subset of $S$ and $f \in \Lip_0(A)$,
  and show that there exists $g \in Y$ with $\|g\| = \|f\|$ and $g(x) = f(x)$ for all $x \in A$.
  Find the minimal $N$ so that $A \subset S_N$.
  Extend $f$ to $\hat{f}$ on $S_N$ by McShane--Whitney.
  Define $g = \hat{f} \circ r_N \in \Lip_0(M)$.
  Then $\|g\| = \|f\|$ and if $r_N(x) = r_N(y)$,
  then $g(x) = g(y)$, so $g \in Y$.

  \ref{item:Y-props2}.
    Let $f\in Y$, let $\varepsilon>0$, and let $N\in \mathbb{N}$ be such that $N\geq 1/\varepsilon$ and for all $x,y\in M$,
    \[
        r_N(x)=r_N(y)\Longrightarrow |f(x)-f(y)|\leq \varepsilon d(x,y).
    \]

    First, we show that $f$ is uniformly locally flat. Let $\delta=\min\{d(x,y)\colon x,y\in S_N,\ x\neq y\}$ and let $x,y\in M$ with $d(x,y)<\delta$. Then
    \[
    d(r_N(x),r_N(y))\leq d(x,y)< \delta.
    \]
    By the definition of $\delta$, we have $r_N(x)=r_N(y)$, which implies
    \[
    |f(x)-f(y)|\leq \varepsilon d(x,y).
    \]

    Next, we show that $f$ is flat at infinity.
    Let $R =\max_{s\in S_N} d(s,0)$.
    For all $z\in M$ with $d(z,0)> R$ and $s\in S_N$, we have $d(s,0)<d(z,0)$, which implies $d(z,s)=d(z,0)$. Therefore, $d(z,S_N)=d(z,0)$, i.e. $r_N(z)=0$.

    Let $x,y\in M$ and assume without loss of generality that $x\not\in B(0,NR)$. If $d(y,0)\leq R$, then $d(x,y)=d(x,0)\geq N R$ and
    \begin{align*}
    |f(x)-f(y)|&\leq |f(x)|+|f(y)|\leq \varepsilon d(x,0)+d(y,0)\leq \varepsilon d(x,y)+R \\
    &\leq \varepsilon d(x,y)+\frac{1}{N}d(x,y)=  (\varepsilon+\frac{1}{N}) d(x,y)\leq 2\varepsilon d(x,y).
    \end{align*}
    If $d(y,0)>R$, then $r_N(y)=0=r_N(x)$, and therefore,
    \[
    |f(x)-f(y)|\leq \varepsilon d(x,y). \qedhere
    \]
\end{proof}

Our next aim is to show that every $f \in Y$
attains its norm when $M$ is spherically complete.
The proof depends on a few lemmas which we will prove first.
We know that for $f \in S_Y$,
there exist sequences $(x_k)$ and $(y_k)$ in $M$ satisfying
\begin{equation*}
  \lim_k \frac{f(x_k) - f(y_k)}{d(x_k,y_k)} = 1.
\end{equation*}
Since $f$ is flat at infinity, both
$(x_k)$ and $(y_k)$ can be assumed to be bounded.
The next lemma connects $f$'s values on $x_k$
with its values on $r_{n_k}(x_k)$ for an increasing sequence
$(n_k)$ of natural numbers.
Note that we do not know if $(r_{n_k}(x_k))$ converges
in this lemma.
The subsequent lemmas will use spherical completeness
to extract a subsequence of $(x_k)$ so that we have convergence.
This allows us to replace $(x_k)$ and $(y_k)$ with sequences
of elements from $S$ and show norm-attainment of $f$.

\begin{lem}\label{lem: f(x_k) converges if and only if f(r_n(x_k)) converges}
  Let $f\in Y$, let $(x_k)$ be a bounded sequence in $M$,
  and let $(n_k)$ be a strictly increasing sequence in $\mathbb{N}$.
  Then
  \begin{equation*}
    \lim_k \big(f(x_k)-f(r_{n_k}(x_{k}))\big)=0.
  \end{equation*}
\end{lem}

\begin{proof}
  Let $R= \sup_{k\in \mathbb{N}} d(x_k,s_1)$, let $\varepsilon>0$,
  and let $N\in \mathbb{N}$ be such that
  \begin{equation*}
    r_N(x)=r_N(y)\Longrightarrow |f(x)-f(y)|\leq \frac{\varepsilon}{R}.
  \end{equation*}
  For $k\geq N$, we have $r_N(r_{n_k}(x_k))=r_N(x_k)$, and therefore,
  \begin{equation*}
    |f(x_k)-f(r_{n_k}(x_k))|\leq \frac{\varepsilon}{R}
    d(x_k,r_{n_k}(x_k))\leq \frac{\varepsilon}{R} d(x_k,s_1)\leq
    \varepsilon.
    \qedhere
  \end{equation*}
\end{proof}

\begin{lem}\label{lem: s_n distances decreasing}
    Let $M$ be spherically complete and let $(s_{n_j})$ be a pseudo-Cauchy subsequence of $S$ satisfying $r_{n_j}(s_{n_{j+1}})=s_{n_j}$ for all $j\in \mathbb{N}$. Then $(s_{n_j})$ is Cauchy.
\end{lem}

\begin{proof}
  Since $M$ is spherically complete, there exists by
  Proposition~\ref{prop:SpComplete=PseudoCauchyLimit}
  a pseudo-limit $x$ of $(s_{n_j})$.
  For all $j\in \mathbb{N}$, we have
  $d(x,s_{n_{j+1}}) < d(x,s_{n_j})$, and therefore,
  $d(x,s_{n_j})=d(s_{n_j},s_{n_{j+1}})$.
  
  Let $\varepsilon > 0$.
  Since $S$ is dense in $M$, we can
  find $N\in \mathbb{N}$ with $d(s_N,x) < \varepsilon$.
  For all $n_j > N$, we have $r_{n_j}(s_{n_{j+1}})=s_{n_j}$ 
  by our assumption and hence
  \begin{align*}
    d(s_{n_j},s_{n_{j+1}})
    &< d(s_N,s_{n_{j+1}})
      \le \max\{d(s_N,x),d(x,s_{n_{j+1}})\} \\
    &\leq \max\{\varepsilon, d(s_{n_j},s_{n_{j+1}})\}
      = \varepsilon.
  \end{align*}
  This shows that $\lim_j d(s_{n_j},s_{n_{j+1}}) = 0$,
  hence $(s_{n_j})$ is Cauchy.
\end{proof}

\begin{lem}\label{lem: subsequence comverge}
    Let $M$ be spherically complete and let $(x_k)$ be a sequence in $M$. There exists a subsequence $(x_{k_n})$ of $(x_k)$ such that $(r_n(x_{k_n}))_{n\in \mathbb{N}}$ converges and $r_l(x_{k_n})=r_l(x_{k_l})$ for all $l\leq n$.
\end{lem}
\begin{proof}
    The idea of the proof is to construct subsequences $(x^n_{k})_{k\in \mathbb{N}}$ 
    of $(x_k)_{k\in \mathbb{N}}$ inductively by $n\in \mathbb{N}$,
    and then show that $(x^n_n)_{n\in \mathbb{N}}$ 
    is a subsequence of $(x_k)$ that satisfies the conditions in the lemma, i.e. $(r_n(x^n_n))_{n\in \mathbb{N}}$ converges and $r_l(x_n^n)=r_l(x^l_l)$ for all $l\leq n$.
    
    First, let $x^1_k=x_k$ for all $k\in \mathbb{N}$ and note that $r_1(x^1_k)=s_1$. 
    
    Fix $n\in \mathbb{N}$ and assume that 
    for all $l\in \mathbb{N}$ with $2\leq l\leq n$, 
    there is a subsequence $(x^l_k)_{k\in \mathbb{N}}$ 
    of $(x^{l-1}_k)_{k\in \mathbb{N}}$ satisfying 
    $r_l(x^l_k)=r_l(x^l_l)$ for all $k\in \mathbb{N}$.
    
    Since $S_{n+1}$ is finite, 
    there exists a subsequence $(x^{n+1}_k)_{k\in \mathbb{N}}$
    of $(x^n_k)_{k\in \mathbb{N}}$ such that
    $r_{n+1}(x^{n+1}_k)=r_{n+1}(x^{n+1}_1)$ for all $k\in \mathbb{N}$. 
    For any $l\in \mathbb{N}$ with $1\leq l\leq n+1$ and $k\in \mathbb{N}$, 
    $r_l(x^{n+1}_k)=r_l(x^l_l)$ because $(x^{n+1}_k)_{k\in \mathbb{N}}$ 
    is a subsequence of $(x^l_k)_{k\in \mathbb{N}}$. 
    
    It remains to show that the sequence 
    $(y_n) = (r_n(x^n_{n}))$ converges. 
    This is clear if $(y_n)$ is eventually constant, 
    so assume that this is not the case. 
    
    Let $(y_{n_i})$ be a subsequence of $(y_n)$ 
    satisfying $y_{n_i}\neq y_{n_{i+1}}$ 
    for all $i\in \mathbb{N}$. 
    Note that for all $i,j\in \mathbb{N}$, $i<j$, we have
    \[
    r_{n_i}(y_{n_j})=r_{n_i}(r_{n_j}(x^{n_j}_{n_j}))=r_{n_i}(x^{n_j}_{n_j})=r_{n_i}(x^{n_i}_{n_i})=y_{n_i},
    \]
    and therefore, for $i,j,k\in \mathbb{N}$ with $i<j<k$, we have 
    \[
    r_{n_i}(y_{n_k})=y_{n_i}\neq y_{n_j}= r_{n_j}(y_{n_k}),
    \]
    which implies
    \[
    d(y_{n_i},y_{n_k})=d(r_{n_i}(y_{n_k}),y_{n_k})>d(r_{n_j}(y_{n_k}),y_{n_k})=d(y_{n_j},y_{n_k}).
    \]
    Thus, $(y_{n_i})$ is a pseudo-Cauchy sequence with $r_{n_i}(y_{n_{i+1}}) = y_{n_i}$.
    By Lemma \ref{lem: s_n distances decreasing}, 
    $(y_{n_i})$ is Cauchy, so $(y_n)$ converges.
\end{proof}

\begin{prop}\label{prop: norm-attainment} If $M$ is spherically complete, then every $f \in Y$ attains its norm.
\end{prop}
\begin{proof}
  Let $f \in S_Y$ and find $(x_k)$ and $(y_k)$ in $M$ satisfying
  \begin{equation*}
    \lim_k \frac{f(x_k) - f(y_k)}{d(x_k,y_k)} = 1.
  \end{equation*}
  By Proposition~\ref{prop:Y-closed_sep_subspace_of_lip0u}~\ref{item:Y-props2},
  $(x_k)$ and $(y_k)$ are bounded and there exists $\delta>0$ and $N\in \mathbb{N}$ such that $d(x_k,y_k)\geq \delta$ for all $k\geq N$. By Lemma~\ref{lem: subsequence comverge}, there exists a subsequence $(x_{k_l})$ of $(x_k)$ such that $r_l(x_{k_l})$ converges and $r_m(x_{k_l})=r_m(x_{k_m})$ for all $m\leq l$. Furthermore, there exists a subsequence $(y_{k_{l_m}})$ of $(y_{k_l})$ such that $r_m(y_{k_{l_m}})$ converges. Let $x=\lim_l r_{l}(x_{k_{l}})$ and $y=\lim_m r_{m}(y_{k_{l_m}})$. By Lemma~\ref{lem: f(x_k) converges if and only if f(r_n(x_k)) converges},
\begin{align*}
    \lim_m f(x_{k_{l_m}})=\lim_m f(r_{l_m}(x_{k_{l_m}}))= f(x)\quad \text{and} \quad \lim_m f(y_{k_{l_m}})=f(y).
\end{align*}
Therefore, 
\[
0<\delta\leq \lim_m d(x_{k_{l_m}},y_{k_{l_m}})=\lim_m
    \big(f(x_{k_{l_m}})-f(y_{k_{l_m}})\big)=f(x)-f(y)\leq d(x,y).
\]
On the other hand,
  \begin{align*}
  f(x)-f(y)&=\lim_m d(x_{k_{l_m}}, y_{k_{l_m}})\geq \lim_m d\bigl(r_m(x_{k_{l_m}}),r_m(y_{k_{l_m}})\bigr)\\
  &=\lim_m d\bigl(r_m(x_{k_{m}}),r_m(y_{k_{l_m}})\bigr)\\
  &=d(x,y).
  \end{align*}
  Therefore, $f(x)-f(y)=d(x,y)$.
\end{proof}

We have already noted in Corollary~\ref{cor:proper_is_dual}
that our method can recover Dalet's result \cite{MR3376824}
that for proper ultrametric spaces $M$ the free space
$\mathcal{F}(M)$ is a dual.
Dalet showed that $\lip_0^u(M)$ is a predual and
next we will show that actually our predual $Y$
is identical to $\lip_0^u(M)$ in this case.

\begin{lem}\label{lem: Y is lip_0 M if M is proper}
  We have $\lip^u_0(M) = Y$ whenever
  $M$ is a proper ultrametric space.
\end{lem}

\begin{proof}
  We have already observed that every $f \in Y$
  is uniformly locally flat and flat at infinity
  in Proposition~\ref{prop:Y-closed_sep_subspace_of_lip0u}~\ref{item:Y-props2}.

  To finish the proof, it is enough to let
  $f \in \lip_0^u(M)$ and show that $f \in Y$.
  Let $\varepsilon > 0$.
  Find $R > 0$ so that
  \begin{equation*}
    |f(x) - f(y)| < \varepsilon d(x,y)
  \end{equation*}
  whenever $x$ or $y$ are not in $B[0,R]$ and $x \neq y$.
  Since $f$ is (uniformly) locally flat,
  there exists $\delta > 0$ such that
  $|f(x) - f(y)| < \varepsilon d(x,y)$
  whenever $d(x,y) < \delta$.

  By compactness there exists a finite covering
  $B(x_1,\delta), \ldots, B(x_K,\delta)$ of $B[0,R]$.
  Since $M$ is ultrametric, we may assume
  $B(x_i,\delta) \cap B(x_j,\delta) = \emptyset$ for $i \neq j$.
  Choose $N \ge K$ so that for each $k \le K$ there exists
  $1 \le j \le N$ with $s_j \in B(x_k,\delta)$.
  This ensures that if $r_N(x) = r_N(y)$ for $x,y \in B[0,R]$,
  then $x,y \in B(x_i,\delta)$ for some $i$ and hence
  $|f(x) - f(y)| \le \varepsilon d(x,y)$
  from the flatness of $f$.

  On the other hand, if $r_N(x) = r_N(y)$ and either
  $x$ or $y$ is not in $B[0,R]$, then
  $|f(x) - f(y)| \le \varepsilon d(x,y)$
  from the choice of $R$.
  Consequently, we get $f \in Y$.
\end{proof}
In Corollary~\ref{cor: ultrametric proper equiv}, 
we show that the reverse implication of the previous lemma also holds, 
i.e. $Y=\lip_0^u(M)$ if and only if $M$ is proper. 

A few examples of non-proper ultrametric spaces,
for which the Lipschitz-free space is a dual,
are previously known from literature. 
For instance, consider $M = \mathbb{N} \cup \{0\}$ with discrete metric, $d(n,m) = 1$ for $n \neq m$.
Obviously $M$ is not proper, but in the language of Weaver,
$M$ is ``rigidly locally compact'' and has the ``separation property'',
so the duality of $\mathcal{F}(M)$ was also known in this case
(see \cite[Corollary~5.5]{MR1386169}), with 
\begin{equation*}
  Z = \{ f \in \Lip_0(M) : \lim_n f(n) = 0 \}
\end{equation*}
as a predual for $\mathcal{F}(M)$.
It is clear that $M$ is a spherically complete ultrametric space
and it is not too difficult to show that
our predual $Y$ is exactly $Z$.

\subsection{\texorpdfstring{$M$}{M}-ideals of the space of Lipschitz functions}

In \cite{MR3376824}, it was shown that if $M$ is proper, then for any $\varepsilon>0$, the space $\lip_0^u(M)$ is $(1+\varepsilon)$-isometric to a subspace of $c_0$, which means that it is $M$-embedded. We show that $\lip_0^u(M)$ is an $M$-ideal in $\Lip_0(M)$ for any ultrametric space $M$. However, we first note that this is not the case for metric spaces in general. Namely, we give an example of a compact metric space $M$ for which $\lip_0^u(M)$ is not an $M$-ideal of $\Lip_0(M)$, thereby answering a question by Werner~\cite{MR4472717}. 

\begin{exmp}\label{exmp:noWerner}
    Let $M = [0,1]\cup\{p\}$ be a metric space where 
    \[
    d(x,y)=\begin{cases}
        |x-y|, & \text{if $x,y\in [0,1]$;}\\
        1/2, & \text{if $x=p$ or $y=p$, and $x\neq y$}.
    \end{cases}
    \]
    Then $\lip_0^u(M)$ is not an $M$-ideal in $\Lip_0(M)$.
\end{exmp}

\begin{proof}
Since $\lip_0^u(M)$ is the range of the
  contractive projection $Q$ on $\Lip_0(M)$ defined by
  \[
    Qf(x)
    =
    \begin{cases}
      0, & \text{if $x \in [0,1]$;}\\
      f(x), & \text{if $x = p$},
    \end{cases}
  \]
  we get from \cite[Proposition~I.1.2 and Corollary~I.1.3]{MR1238713}
  that $\lip_0^u(M)$ is an $M$-ideal in $\Lip_0(M)$ if and only if $Q$ satisfies
  \[
    \|f\|
    = \max\{\|Q(f)\|, \|f - Q(f)\|\},
  \]
  for all $f \in \Lip_0(M).$
  However, putting
  \begin{equation*}
    F(x)=\begin{cases}
      x, & \text{if $x\in [0,1]$;}\\
      1/2, &\text{if $x=p$},
    \end{cases}
  \end{equation*}
  and observing
  \begin{align*}
   \|F\|
   = 1 
    < 2
    = \|F - Q(F)\|
    = \max\{\|Q(F)\|, \|F - Q(F)\|\},
  \end{align*}
  we see that $\lip_0^u(M)$ is not an $M$-ideal in $\Lip_0(M)$.
\end{proof}

Let $Z$ be a subspace of a Banach space $X$.
Recall the \emph{3-ball property} characterisation
of $M$-ideals (see e.g. \cite[Theorem~I.2.2]{MR1238713})
which says that $Z$ is an $M$-ideal of $X$ if and only
if for all $z_1, z_2, z_3 \in B_Z,$ all $x \in B_X,$ and all
$\varepsilon > 0$, there is $z \in Z$ such that
\[
  \|z_i + x - z\|
  \le 1 + \varepsilon
\]
for all $i = 1, 2, 3.$

\begin{prop}\label{prop: unif locally flat M-ideal in Lip_0(M)}
    Let $M$ be an ultrametric space, then the space $\lip^u_0(M)$ is an $M$-ideal in $\Lip_0(M)$.
\end{prop}

\begin{proof}
  Let $\varepsilon > 0$, let $f_1,f_2,f_3 \in B_{\lip_0^u(M)}$,
  and let $F \in B_{\Lip_0(M)}$.
  It is enough to find $h \in \lip_0^u(M)$ such that for all $i \in \{1,2,3\}$,
  \begin{equation*}
    \|f_i+F - h \| \le 1 + 2\varepsilon.
  \end{equation*}
  Let $q \in (0,1)$.
  Since the $f_i$'s are uniformly locally flat and flat at infinity,
  there exists $\delta=q^N$ for some $N\in \mathbb{N}$
  such that for all $i\in \{1,2,3\}$, we have
  $|f_i(x) - f_i(y)| < \varepsilon d(x,y)$
  whenever $d(x,y) < \delta$, $d(x,0) > R$, or $d(y,0) > R$.
  We can and will assume that
  $\delta \leq  \varepsilon\leq R$.

  By Lemma~\ref{lem: partitioning}, there exists
  a sequence of partitions $(C_n)_{n \in \mathbb{Z}}$ of $B[0,R]$.
  Define $h\in \lip_0^u(M)$ by setting for all $x\in C_{2N}$, 
  \begin{equation*}
    h|_{B(x,q^{2N})}
    =
    \frac{
      \sup_{z\in B(x,q^{2N})}F(z)
      +
      \inf_{z\in B(x,q^{2N})}F(z)
    }{2},
  \end{equation*}
  and $h(x)=0$ if $x \in M\setminus B[0,R]$.
  Note that for all $z\in B[0,R]$,
  \[
  |F(z) - h(z)|\leq q^{2N}/2=\delta^2/2.
  \]
  Let $i\in \{1,2,3\}$ and let $x,y\in M$. To estimate the norm $\|f_i+F - h \|$, we have to look at a few cases.
  
  \noindent
  \textbf{Case I.}
  If either $x,y \in B[0,R]$ with $d(x,y) < q^{2N}$, or
  $x,y\in M\setminus B[0,R]$,
  then $h(x) = h(y)$ and $|f_i(x) - f_i(y)| \le \varepsilon d(x,y)$,
  and hence
  \begin{equation*}
    |f_i(x) + F(x) - h(x) - f_i(y) - F(y) + h(y)|
    \le (1+\varepsilon)d(x,y).
  \end{equation*}

  \noindent
  \textbf{Case II.}
  If $x\in M\setminus B[0,R]$ and
  $y\in B[0,R]$, then $d(x,0)>d(y,0)$,
  which implies $d(x,y)=d(x,0)\geq R$.
  Therefore,
  \begin{align*}
    &|f_i(x) + F(x) - h(x)-f_i(y)-F(y)+h(y)|\\
    &\qquad
    \leq |F(x)|+ |f_i(x)-f_i(y)|+|F(y)-h(y)| \\
    &\qquad
    \leq d(x,0)+\varepsilon d(x,y)+\delta^2 \\
    &\qquad
    \leq d(x,y)+\varepsilon d(x,y)+\varepsilon R
    \leq (1+2\varepsilon)d(x,y).
  \end{align*}

  \noindent
  \textbf{Case III.}
  If $x,y \in B[0,R]$ with $q^{2N}\leq d(x,y)< q^N = \delta$, then
  $|f_i(x) - f_i(y)| \le \varepsilon d(x,y)$
  and
  $|F(x) - h(x)| < q^{2N}/2$ (and similarly for $y$).
  Hence,
  \begin{align*}
    &|f_i(x) + F(x) - h(x) - f_i(y) - F(y) + h(y)| \\
    &\qquad
    \leq \varepsilon d(x,y)+ q^{2N}/2+q^{2N}/2
    \le (1+\varepsilon)d(x,y).
  \end{align*}

  \noindent
  \textbf{Case IV.}
  If $x,y \in B[0,R]$ with $d(x,y)\geq \delta$, then
  again $|F(x) - h(x)| < q^{2N}/2$ (and similarly for $y$),
  \begin{align*}
    &|f_i(x)+F(x)  - h(x) - f_i(y)- F(y) + h(y)| \\
    &\qquad
    \leq d(x,y)+ q^{2N}/2+q^{2N}/2
    = d(x,y)+\delta^2\\
    &\qquad
    \leq d(x,y)+\delta d(x,y)
    \leq (1+\varepsilon) d(x,y).
  \end{align*}
  Thus, we have shown $\|f_i + F - h\| \le 1 + 2\varepsilon$.
\end{proof}

We finish the section by showing that for
a spherically complete $M$,
the predual $Y$ of $\mathcal{F}(M)$ we constructed
is not $M$-embedded if $M$ is not proper.

\begin{lem}\label{lem: not proper not M-embedded}
    Let $M$ be a separable spherically complete ultrametric space. If $M$ is not proper, then there exist $\delta,R>0$, $N\in \mathbb{N}$, and a subsequence $(s_{n_i})$ of $S$ satisfying for all $i\in \mathbb{N}$, $n_i>N$, $\delta \leq d(s_{n_i},s_{N})\leq R$, and $r_{n_i-1}(s_{n_i})=s_N$.
\end{lem}
\begin{proof}
    Since $M$ is not proper, there exist $\delta>0$ and a bounded sequence $(x_k)$ in $M$ satisfying $\delta\leq d(x_k,x_l)$ for all $k,l\in \mathbb{N}$, $k\neq l$.  By Lemma~\ref{lem: subsequence comverge}, 
    we may further assume, by going to a subsequence if necessary, that $(r_k(x_k))_{k\in \mathbb{N}}$ is a converging sequence satisfying $r_k(x_l)=r_k(x_k)$ for all $l\geq k$.

    We first show that $(r_k(x_k))_{k\in \mathbb{N}}$ is eventually constant. 
    Assume for contra\-diction that this is not the case.
    Then for every $k\in \mathbb{N}$,
    there exists $l>k$ such that $r_{l}(x_{l})\neq r_k(x_k)=r_k(x_{l})$.
    We get
    \[
    d(r_{l}(x_{l}),x_{l})<d(r_k(x_{l}),x_{l})=d(r_k(x_k),x_{l}),
    \]
    and therefore, $d(r_k(x_k),x_{l})=d(r_k(x_k),r_{l}(x_{l}))$.
    Since 
    \[
    \lim_{k,l\to \infty} d(r_k(x_k),r_{l}(x_{l}))=0,
    \]
    we get $\lim_l r_l(x_l) = \lim_l x_l$,
    which is a contradiction because $(x_l)$ does not converge.

    Since $(r_k(x_k))$ is eventually constant, there exists $K\in \mathbb{N}$ so that $r_k(x_k)=r_K(x_K)$ for all $k\geq K$. Denote $y=r_K(x_K)\in S_K$ and $R=\sup_{k\in \mathbb{N}}d(x_k,y)$. Note that $R<\infty$ because $(x_k)$ is bounded. There exists at most one $k> K$ satisfying $d(x_k, y)<\delta$ because for any $k,l\in \mathbb{N}$ with $k\neq l$, we have
    \[
    \delta \leq d(x_k,x_l)\leq \max\{d(x_k,y),d(x_l,y)\}.
    \]

    Fix $k_1>K$ so that $d(x_{k_1},y)\geq\delta$. 
    Since $S$ is dense in $M$, there exists 
    \[
    n_1=\min\{n>K\colon r_n(x_{k_1})\neq y\}=\min\{n>K\colon r_n(x_{k_1})=s_n\}.
    \]
    Since $r_{k_1}(x_{k_1})=y$, we have $n_1>k_1$. From
    $d(s_{n_1},x_{k_1})<d(x_{k_1},y)$ we get
    $d(s_{n_1},y)=d(x_{k_1},y)$.
    Therefore,
    \[
    \delta \leq d(s_{n_1},y)\leq R.
    \]
    
    Fix $k_2>n_1$ so that $d(x_{k_2},y)\geq \delta$ and let
    \[
    n_2=\min\{n>K\colon r_n(x_{k_2})\neq y\}=\min\{n>K\colon r_n(x_{k_2})=s_n\}.
    \]
    Similarly to before, we get $n_2>k_2$ and $d(s_{n_2},y)=d(x_{k_2},y)$, which implies
    \[
    \delta \leq d(s_{n_2},y)\leq R.
    \]
    Continuing iteratively, we get $K<k_1<n_1<k_2<\dotsb$ such that for all $i\in \mathbb{N}$, $r_{n_i-1}(s_{n_i})=y$ and 
    $\delta \leq d(s_{n_i},y)\leq R.$ This completes the proof.
\end{proof}

To check that the predual $Y$ of $\mathcal{F}(M)$ is not an $M$-ideal
in $\Lip_0(M)$ for the discrete metric space $M=\mathbb{N}\cup \{0\}$,
it suffices to define $f\in S_Y$ by
\[
f(n)=\begin{cases}
    1, &\text{if $n=1$};\\
    0, &\text{otherwise,}
\end{cases}
\]
define $F\in S_{\Lip_0(M)}$ by setting $F(n)=1$ for all $n\in \mathbb{N}$,
and calculate $\|\pm f + F - g\|$ for $g \in Y$ to see that $Y$ fails the 3-ball property.
The proof for the following proposition is a generalisation 
of this idea.

% Actual calc: ||f + F - g|| >= 2 - g(1) (on m_10)
% and ||-f + F - g|| >= 1 + g(1) (on m_n,1 and n->infty)
% Add the inequalities to get \varepsilon\geq 1/2
\begin{prop}\label{prop: proper and M-embedded equivalence}
    Let $M$ be a separable spherically complete ultrametric space. The predual $Y$ of $\mathcal{F}(M)$ is $M$-embedded if and only if $M$ is proper.
\end{prop}

\begin{proof}
  If $M$ is proper, then $Y$ is $M$-embedded by \cite[Lemma~3.9]{MR3376824} and Lemma~\ref{lem: Y is lip_0 M if M is proper}.
  Assume that $M$ is not proper.
    By Lemma~\ref{lem: not proper not M-embedded}, there exist $\delta,R>0$, $N\in \mathbb{N}$, and a subsequence $(s_{n_i})$ of $S$ satisfying for all $i\in \mathbb{N}$, $n_i>N$, $\delta \leq d(s_{n_i},s_{N})\leq R$ and $r_{n_i-1}(s_{n_i})=s_N$.
    
    For every $i\in \mathbb{N}$, define
    \[
    f_i(x)=
    \begin{cases}
    d(s_{n_i},s_N), & \text{if $r_{n_i}(x)=s_{n_i}$;}\\
    0, & \text{otherwise.}
    \end{cases}
    \]
    We show that $f_i\in S_Y$. To see that $f_i\in \Lip_0(M)$ with $\|f_i\|\leq 1$, 
    we take $x,y\in M$ such that $f_i(x)\neq f_i(y)$ 
    (by symmetry, we may assume that $r_{n_i}(x)=s_{n_i}$ and $r_{n_i}(y)\neq s_{n_i}$), 
    and note that
    \[
    |f_i(x)-f_i(y)|=d(s_{n_i},s_N)=  d(s_{n_i}, S_{n_i-1})\leq d(r_{n_i}(x),r_{n_i}(y))\leq d(x,y).
    \]
    It is clear that $f_i\in Y$ because if $r_{n_i}(x)=r_{n_i}(y)$, then $f_i(x)=f_i(y)$. Finally, we have $\|f_i\|\geq 1$ because $f_i(s_{n_i})-f_i(s_N)=d(s_{n_i},s_N)$.

    Define $F\in \Lip_0(M)$ by $F(x)=\sup_{i\in \mathbb{N}} f_i(x)$ and note that $\|F\|\leq 1$ by \cite[Proposition~1.32]{MR3792558}. For all $i\in \mathbb{N}$, we have $F(s_{n_i})\geq d(s_{n_i},s_N)$ and $F(s_N)=0$, so $F(s_{n_i})=d(s_{n_i},s_N)$.

    Let $g\in Y$ and let $\varepsilon>0$. Assume that
    \[
    \|\pm f_1+F-g\|\leq 1+\varepsilon.
    \]
    We will show that $\varepsilon$ is bounded away from $0$
    and hence $Y$ does not satisfy the 3-ball property.
    
    Note that
    \begin{align*}
        (1+\varepsilon)d(s_{n_1},s_N)&\geq (f_1+F-g)(s_{n_1})-(f_1+F-g)(s_N)\\
        &=2d(s_{n_1},s_N)-g(s_{n_1})+g(s_N),
    \end{align*}
    which implies $g(s_{n_1})-g(s_N)\geq (1-\varepsilon)d(s_{n_1},s_N)$.
    Since $g\in Y$, there exists $K>N$ satisfying for all $x\in M$ with $r_K(x)=s_N$,
    \[
    |g(x)-g(s_N)|\leq \varepsilon d(x, s_N).
    \]
    Fix $i\in \mathbb{N}$ such that $n_i> K$. Then $|g(s_{n_i})-g(s_N)|\leq \varepsilon d(s_{n_i},s_N)$ because
    \[
    r_K(s_{n_i})=r_K(r_{n_i-1}(s_{n_i}))=r_K(s_N)=s_N.
    \]
    We have
    \begin{align*}
      (1+\varepsilon)d(s_{n_1},s_{n_i})
      &\geq
      (-f_1+F-g)(s_{n_i}) - (-f_1+F-g)(s_{n_1})\\
      &=
      F(s_{n_i}) + g(s_N) - g(s_{n_i}) + g(s_{n_1}) - g(s_N)\\
      &\ge
      d(s_{n_i},s_N) - \varepsilon d(s_{n_i},s_N) +
      (1-\varepsilon) d(s_{n_1}, s_N) \\
      &=
      (1-\varepsilon) \bigl(d(s_{n_1},s_{N}) + d(s_{n_i},s_N) \bigr)\\
      &\geq
      (1-\varepsilon) \bigl(
      \max\{d(s_{n_1},s_{N}), d(s_{n_i},s_N)\} + \delta
      \bigr)\\
      &
      \geq (1-\varepsilon)\bigl( d(s_{n_1},s_{n_i}) + \delta \bigr).
    \end{align*}
        Therefore,
        \[
        2\varepsilon d(s_{n_1},s_{n_i})\geq (1-\varepsilon)\delta\geq \delta - \varepsilon d(s_{n_1},s_{n_i}),
        \]
        so $3\varepsilon d(s_{n_1},s_{n_i})\geq \delta$.
        Since
        \[
        d(s_{n_1},s_{n_i})\leq \max\{d(s_{n_1},s_N), d(s_N,s_{n_i})\}\leq R,
        \]
        we get $\varepsilon\geq \frac{\delta}{3R}$. Therefore, $Y$ is not $M$-embedded.
\end{proof}

We end this section with a corollary that connects a few threads
to give a characterisation of when a separable spherically complete ultrametric space
is proper in terms of our predual $Y$.

\begin{cor}\label{cor: ultrametric proper equiv}
    Let $M$ be a separable spherically complete ultrametric space. The following are equivalent:
    \begin{enumerate}
        \item\label{item: cor_i} $M$ is proper;
        \item\label{item: cor_ii} $Y=\lip_0^u(M)$;
        \item\label{item: cor_iii} $Y$ is $(1+\varepsilon)$-isometric to a subspace of $c_0$ for any $\varepsilon>0$;
        \item\label{item: cor_iv} $Y$ is $M$-embedded.
    \end{enumerate}
\end{cor}
\begin{proof}
    \ref{item: cor_i}$\Rightarrow$\ref{item: cor_ii} holds 
    by Lemma~\ref{lem: Y is lip_0 M if M is proper}. 
    $M$-embedded preduals are unique by \cite[Proposition~IV.1.9]{MR1238713}, 
    so \ref{item: cor_ii}$\Rightarrow$\ref{item: cor_i} 
    follows from Propositions~\ref{prop: unif locally flat M-ideal in Lip_0(M)} 
    and~\ref{prop: proper and M-embedded equivalence}. 
    \ref{item: cor_i}$\Rightarrow$\ref{item: cor_iii} now follows from \cite[Lemma~3.9]{MR3376824}, 
    while~\ref{item: cor_iii}$\Rightarrow$\ref{item: cor_iv} 
    holds even if we replace $Y$ with a general Banach space by the 3-ball property.
    Finally, \ref{item: cor_i}$\Leftrightarrow$\ref{item: cor_iv} 
    holds by Proposition~\ref{prop: proper and M-embedded equivalence}.
\end{proof}

\section{Strongly extreme points in spaces of Lipschitz functions}
\label{sec:strongly-extr-points}
Let $X$ be a Banach space.
Recall that $x \in S_X$ is a \emph{strongly extreme point}
if $\lim_n \|x \pm y_n\| = 1$ for a sequence $(y_n)$ in $X$
implies $\lim_n \|y_n\| = 0$.
Another way of saying this is that the norm is
midpoint locally uniformly rotund at $x$.
At the other end of the scale we have the property that
$X$ is \emph{locally almost square} (LASQ); that is,
if for every $x \in S_X$ and $\varepsilon > 0$, there exists
$y \in S_Y$ such that $\|x \pm y\| \le 1 + \varepsilon$.

Gar{\'i}ca-Lirola and Rueda Zoca proved in
\cite[Proposition~4.1]{MR3619230} that $\Lip_0(M)$ is not LASQ
if $M$ is a bounded uniformly discrete metric space.
In fact, what they prove is that if $f(x) = d(x,0)$
and $\|f \pm g_n\| \to 1$, then $\|g_n\| \to 0$.
Hence, $f$ is strongly extreme.
In particular, this shows that for the ultrametric space
$M = \mathbb{N} \cup \{0\}$ with discrete metric, $\Lip_0(M)$ is not LASQ.
In the same paper \cite[Remark~3.5]{MR3619230} it was remarked
that if $M$ is an ultrametric space that
has a limit point, then $\lip_0^u(M)$ ($S(M)$ in their notation)
is not only locally almost square,
but even unconditionally almost square.

It is not too difficult to show that in general
if $M$ is an ultrametric space, then $f(x) = d(x,0)$ is
not necessarily a strongly extreme point of $B_{\Lip_0(M)}$.
The aim of this section is to prove that we can
always find a strongly extreme point in the unit ball of $\Lip_0(M)$.

\begin{thm}\label{thm:strongly_extreme_and_LASQ}
  Let $M$ be an ultrametric space.
  Then $B_{\Lip_0(M)}$ contains a strongly extreme point
  and in particular, $\Lip_0(M)$ is not LASQ.
\end{thm}
\noindent
We have broken the proof of the theorem
into a series of lemmas.

Throughout the rest of this section let $M$ be an ultrametric space
and let $q\in (0,1)$.
Fix partitions $(C_n)_{n\in\mathbb{Z}}$ of $M$
satisfying the conditions of Lemma~\ref{lem: partitioning}.
From Lemma~\ref{lem: partitioning} \ref{item:c} and \ref{item:d}
we have that for all $n\in \mathbb{Z}$ and $x\in M$,
there exists a uniquely defined element $y\in C_n$
satisfying $d(x,y)=d(x,C_n)< q^n$.
Therefore, for all $n\in \mathbb{Z}$,
we can define a function $\phi_n\colon M\to C_n$ by
\begin{equation*}
    d(x, \phi_n(x))=d(x,C_n)< q^n\qquad \text{for all $x\in M$}.
\end{equation*}
We list and prove some properties of these functions.

\begin{lem}\label{lem: ultrametric phi_n properties}
    Let $x,y\in M$ and let $n\in \mathbb{Z}$.
    \begin{enumerate}
    \item\label{item:phi_n-P1}
      $\phi_n\circ \phi_{n+1}=\phi_n$.
    \item\label{item:phi_n-P2}
      $\phi_n(x)=\phi_n(y)$ if and only if
      \begin{equation*}
        d(x,y) \le \max\{d(\phi_n(x),x), d(\phi_n(y),y)\}.
      \end{equation*}
    \item\label{item:phi_n-P3}
      If either $d(\phi_n(x),x)\leq d(x,y)$ or $d(\phi_n(y),y)\le d(x,y)$,
      then
      \[
        \max\{d(\phi_n(x),x),d(\phi_n(y),y)\} \le d(x,y).
      \]
    \item\label{item:phi_n-P4}
      If $\phi_n(x)\neq\phi_n(y)$, then 
      \[
        d(\phi_n(x),y)= d(\phi_n(y),x)=d(x,y)\geq q^{n}.
      \]
    \item\label{item:phi_n-P5}
      We have $d(\phi_n(x),x) = d(\phi_n(y),y)$ whenever
      \begin{equation*}
        d(x,y) < \max\{d(\phi_n(x),x),d(\phi_n(y),y)\}.
      \end{equation*}
    \end{enumerate}
\end{lem}
\begin{proof}
    \ref{item:phi_n-P1}. For all $x\in M$,
    \[
      d(\phi_{n+1}(x), \phi_n(x))
      \leq \max\{d(\phi_{n+1}(x),x), d(\phi_n(x),x)\}< q^n,
    \]
    and therefore, $\phi_n(\phi_{n+1}(x))=\phi_n(x)$.
    
    \ref{item:phi_n-P2}. If $\phi_n(x) = \phi_n(y)$, then
    \begin{equation*}
      d(x,y) \le
      \max\{d(\phi_n(x),x),d(\phi_n(x),y)\}
      =
      \max\{d(\phi_n(x),x),d(\phi_n(y),y)\}.
    \end{equation*}
     For the converse, we may assume $d(\phi_n(x),x) \ge d(x,y)$. 
     Then
    \begin{equation*}
      d(\phi_n(x),y) \le \max\{d(\phi_n(x),x),d(x,y)\}
      = d(\phi_n(x),x) < q^{n},
    \end{equation*}
    so $\phi_n(y) = \phi_n(x)$.

    \ref{item:phi_n-P3}. Without loss of generality, we may assume
    that $d(\phi_n(x),x) \le d(x,y)$ since everything
    is symmetric in $x$ and $y$.
    We have $\phi_n(x),\phi_n(y)\in C_n$, so
    \begin{equation*}
      d(\phi_n(y),y) \le
      d(\phi_n(x),y) \le \max\{d(\phi_n(x),x),d(x,y)\} = d(x,y).
    \end{equation*}
    
    \ref{item:phi_n-P4}.
    If $\phi_n(x)\neq \phi_n(y)$, then by definition of $\phi_n$,
    we have
    $d(\phi_n(x),y)\geq q^n$ and $d(\phi_n(y),x)\geq q^n$.
    By \ref{item:phi_n-P2}, we have
    $d(\phi_n(x),x)<d(x,y)$, which implies 
    \[
    q^n\leq d(\phi_n(x),y)=d(x,y).
    \]
    Similarly, we get $d(\phi_n(y),x)=d(x,y)$.
    
    \ref{item:phi_n-P5}. 
    If
    \begin{equation*}
      d(x,y) < \max\{ d(\phi_n(x),x), d(\phi_n(y),y)\},
    \end{equation*}
    then $\phi_n(x) = \phi_n(y)$ by \ref{item:phi_n-P2}.
    If $d(x,y) < d(\phi_n(y),y)$, we get
    \begin{equation*}
      d(\phi_n(x),x)
      =
      d(\phi_n(y),x)
      =
      \max\{d(\phi_n(y),y), d(x,y)\}
      =
      d(\phi_n(y),y),
    \end{equation*}
    and similarly $d(\phi_n(x),x) = d(\phi_n(y),y)$
    if $d(x,y) < d(x,\phi_n(x))$.
\end{proof}

For $x \in M$, we will focus on the indices $n\in \mathbb{Z}$, where the approximation $\phi_n(x)$ to $x$ changes. We define
\begin{equation*}
  I_x = \{ n \in \mathbb{Z} \colon \phi_n(x)\neq \phi_{n-1}(x) \}.
\end{equation*}

\begin{lem}\label{lem:prop_of_Ix}
  Let $x \in M$.
  The set $I_x$ has the following properties:
  \begin{enumerate}
  \item\label{item:Ix-1}
    $I_x = \emptyset$ if and only if $x = 0$;
  \item\label{item:Ix-2}
    $n \in I_x$ if and only if $q^{n} \leq d(\phi_{n-1}(x),x) < q^{n-1}$;
  \item\label{item:Ix-3}
    for $x \neq 0$, $n_1 = \min I_x$ exists and satisfies
    $q^{n_1} \leq d(x,0) < q^{n_1-1}$;
  \item\label{item:Ix-4}
    $I_x$ is a finite set if and only if $\phi_n(x) = x$
    for some $n\in \mathbb{Z}$.
  \end{enumerate}
\end{lem}

\begin{proof}
  \ref{item:Ix-1}.
  It is clear that $\phi_n(0) = 0$ for all $n \in \mathbb{Z}$.
  If $x \neq 0$, then there exists $n_0 \in \mathbb{Z}$ such
  that $q^{n_0 + 1} \le d(x,0) < q^{n_0}$.
  Then $\phi_m(x) = 0$ for $m \le n_0$ and
  $\phi_m(x) \neq 0$ for $m > n_0$.

  \ref{item:Ix-2}.
  If $d(\phi_{n-1}(x),x) \ge q^n$, then
  $\phi_{n-1}(x) \neq \phi_n(x)$ by the definition of $\phi_n$.
  On the other hand, if $n \in I_x$, then
  $d(\phi_{n-1}(x),x) < q^{n-1}$ by the definition of $\phi_{n-1}$
  and, since $\phi_n(x), \phi_{n-1}(x) \in C_n$
  by Lemma~\ref{lem: partitioning}~\ref{item:b},
  we get
  \begin{align*}
    d(\phi_{n-1}(x),x)
    &= \max\{d(\phi_{n-1}(x),\phi_{n}(x)), d(\phi_{n}(x),x) \} \\
    &= d(\phi_{n-1}(x),\phi_{n}(x)) \ge q^n
  \end{align*}
  by Lemma~\ref{lem: partitioning}~\ref{item:c}.

  \ref{item:Ix-3}.
  Find $n_0$ as in \ref{item:Ix-1}, then use \ref{item:Ix-2} with
  $n_1 = n_0 + 1$ and note that
  $\phi_{n_1 - 1}(x) = \phi_{n_0}(x) = 0$.

  \ref{item:Ix-4}.
  That $I_x$ is finite means that there exists
  $n \in \mathbb{Z}$ such that $\phi_m(x) = \phi_n(x)$
  for all $m \ge n$. This means that
  $d(x,\phi_n(x)) = d(x,\phi_m(x)) < q^m$
  for all $m \ge n$, that is $x = \phi_n(x)$.
\end{proof}

If $x \neq 0$, then
$\min I_x$ exists by Lemma~\ref{lem:prop_of_Ix}, 
and we can assume that
$I_x = \{n_1,n_2,\ldots\}$ is ordered so that
$n_1 < n_2 < \cdots < n_k < \cdots$,
and we denote $i_x(k) = n_k$, the $k$-th smallest number in $I_x$.
If we let $n_0 = n_1 - 1$ and $i_x(0) = n_0$,
we have $\phi_{i_x(0)}(x) = 0$.
For $x = 0$ we can just let $n_0 = 0$ and $i_0(0) = 0$.
This means that for any $x \in M$, we have a well-defined
sequence of elements given by
\begin{equation*}
  \mathfrak{s}(x) :=
  ( \phi_{n_0}(x), \phi_{n_1}(x), \ldots )
  =
  \{ \phi_{i_x(k)}(x)\}_{k=0}^{|I_x|}.
\end{equation*}
We have $\mathfrak{s}(0) = \{0\}$ and for $x \neq 0$,
we have $d(\phi_{n_k}(x),x)<d(\phi_{n_l}(x),x)$ if $0\leq l<k\leq |I_x|$.
Essentially, $\mathfrak{s}(x)$ captures all 
the different values of the sequence 
$(\phi_n(x))_{n\in \mathbb{Z}}$ without repetition.
If $|I_x| < \infty$, then for all $m\geq \max I_x$,
we have $x=\phi_m(x)\in C_m$.
If $I_x$ is an infinite set, then
$\lim_{k} \phi_{n_k}(x) = x$.
Let us also note that
for $0 \le k < |I_x|$,
\begin{equation}
  \label{eq:s_xk}
  \mathfrak{s}(\phi_{i_x(k)}(x))
  = \{\phi_{i_x(l)}(x) \}_{l=0}^k,
\end{equation}
by Lemma~\ref{lem: ultrametric phi_n properties} \ref{item:phi_n-P1}.

\begin{lem}\label{lem: x geometric}
  Let $x \in M$ and denote $(x_k)_{k=0}^{|I_x|} = \mathfrak{s}(x)$.
  Then for all $0 \le K < |I_x|$, we have
  \begin{equation*}
    \sum_{l=K}^{|I_x|} d(x_{l},x)
    \le \frac{1}{q(1-q)} d(x_{K},x).
  \end{equation*}
\end{lem}
  \begin{proof}
    Let $m_x=|I_x|$ and let $0\leq K< m_x$.
    Note that if $I_x$ is finite, 
    then $x_{m_x}=x$. Therefore, we have
  \begin{align*}
    \sum_{l=K}^{m_x} d(x_{l},x)
    &=
      \sum_{l=K}^{m_x-1} d(\phi_{i_x(l)}(x),x)
      =
      \sum_{l=K}^{m_x-1} d(\phi_{i_x(l+1)-1}(x),x) \\
    &<
      \sum_{l=K}^{m_x-1} q^{i_x(l+1)-1}
      \le
      \sum_{l=0}^{\infty} q^{i_x(K+1)+ l-1}
      =
      \frac{q^{i_x(K+1)}}{q(1 - q)}
    \\
    &\le
      \frac{1}{q(1-q)}d(\phi_{i_x(K+1)-1}(x),x)
      =
      \frac{1}{q(1-q)}d(x_{K},x),
  \end{align*}
  where we used Lemma~\ref{lem:prop_of_Ix}~\ref{item:Ix-2}
  for the first and last inequality.
  \end{proof}

\begin{lem}\label{lem: strongly extreme, first K}
  Let $x,y \in M$, $x\neq y$, and let $(x_k)_{k=0}^{|I_x|}=\mathfrak{s}(x)$, $(y_k)_{k=0}^{|I_y|} = \mathfrak{s}(y)$.
  There exists $K \in \mathbb{N}\cup\{0\}$ such that
  \begin{equation*}
    \max\{d(x_{K},x), d(y_{K},y)\} \le d(x,y),
  \end{equation*} 
  and $x_k = y_k$ and
  $d(x_k, x) = d(y_k, y)$ for all $k < K$.
\end{lem}

\begin{proof}
  If $d(\phi_n(x),x)\leq d(x,y)$ for all $n\in \mathbb{Z}$, then
  $d(\phi_n(y),y)\leq d(x,y)$ for all $n\in \mathbb{Z}$ by Lemma~\ref{lem: ultrametric phi_n properties}~\ref{item:phi_n-P3}. Taking $K=0$ and $n\in \mathbb{Z}$ such that $\phi_n(x)=\phi_n(y)=0$, we get
  \[
  \max\{d(x_{K},x),d(y_{K},y)\}=\max\{d(\phi_n(x),x),d(\phi_n(y),y)\}\leq d(x,y).
  \]
  Now assume that there exists
  \begin{equation*}
    N := \min\{ n\in \mathbb{Z}\colon d(\phi_n(x),x)\leq d(x,y)\}.
  \end{equation*}
  By Lemma~\ref{lem: ultrametric phi_n properties}~\ref{item:phi_n-P3},
  we have
  \begin{equation*}
    N = \min\{n\in \mathbb{Z}\colon d(\phi_n(y),y)\leq d(x,y)\}.
  \end{equation*}
  By Lemma~\ref{lem: ultrametric phi_n properties}~\ref{item:phi_n-P2} and~\ref{item:phi_n-P5}, 
  we have $\phi_n(x)=\phi_n(y)$ and $d(\phi_n(x),x)=d(\phi_n(y),y)$ for all $n<N$. 
  Take $K\in \mathbb{N}$ so that $i_x(K)=i_y(K)=N$.
  Then for all $k< K$, we have $x_k=y_k$ and $d(x_k,x)=d(y_k,y)$.
  Furthermore,
  \[
  \max\{ d(x_{K},x),d(y_{K},y)\}= \max\{ d(\phi_{N}(x),x),d(\phi_{N}(y),y)\}\leq d(x,y).
  \qedhere
  \]
  \end{proof}
  
  \begin{lem}\label{lem: strongly extreme, second K}
  Let $x,y \in M$, $x\neq y$, and let $(x_k)_{k=0}^{|I_x|}=\mathfrak{s}(x)$, $(y_k)_{k=0}^{|I_y|} = \mathfrak{s}(y)$.
  There exists $ K \in \mathbb{N}\cup\{0\}$
  such that
  \begin{equation*}
    \max\{d(x_{K},x), d(y_{K},y)\} \le \frac{d(x,y)}{q}
  \end{equation*} 
  and $x_k = y_k$ for all $k \le K$.
  \end{lem}
  \begin{proof}
  Let $N=\min\{n\in \mathbb{Z}\colon \phi_n(x)\neq \phi_n(y)\}$.
  By Lemma~\ref{lem: ultrametric phi_n properties}~\ref{item:phi_n-P2},
  \[
  \max\{d(\phi_{N}(x),x),d(\phi_{N}(y),y)\}< d(x,y).
  \]
  Since $\phi_{N-1}(x)=\phi_{N-1}(y)$, then $\phi_{N}(x)\neq
  \phi_{N-1}(x)$ or $\phi_{N}(y)\neq \phi_{N-1}(y)$.
  Assume that $\phi_{N}(x)\neq \phi_{N-1}(x)$. Then there exists
  $K\in \mathbb{N}$ such that $i_x(K)=N$.
  We have $x_{K-1}=\phi_{N-1}(x)=\phi_{N-1}(y)=y_{K-1}$. Therefore,
  \[
  d(x_{K-1},x)=d(\phi_{N-1}(x),x)< q^{N-1}
  \]
  and
  \[
  d(y_{K-1},y)=d(\phi_{N-1}(y),y)< q^{N-1}.
  \]
  By Lemma~\ref{lem: ultrametric phi_n properties}~\ref{item:phi_n-P4},
  $d(x,y)\geq q^{N}$, so
  \[
    \max\{d(x_{K-1},x), d(y_{K-1},y)\}< q^{N-1}\leq q^{-1}d(x,y).
    \qedhere
  \]
\end{proof}

\begin{proof}[Proof of Theorem~\ref{thm:strongly_extreme_and_LASQ}]
  We will define $f\in \Lip_0(M)$.
  For $x \in M$ with $x \neq 0$,
  denote $(x_k)_{k=0}^{|I_x|} = \mathfrak{s}(x)$,
  and define
  \begin{equation*}
    f(x) = \sum_{k=0}^{m_x} (-1)^{k} d(x_k,x).
  \end{equation*}
  Note that $0 \le f(x) \le d(x_0,x) = d(0,x)$.

  If $x \in M$, $x \neq 0$, and $0 \le k < m_x$,
  then, using \eqref{eq:s_xk} we have
  \[
  f(x_{k})=\sum_{l=0}^{k-1} (-1)^l d(x_l,x_k)
  \quad \text{and}\quad 
  f(x_{k+1})=\sum_{l=0}^{k} (-1)^l d(x_l,x_{k+1}).
  \]
  For all $l< k$, we have
  \[
  d(x_l,x)>d(x_k,x)>d(x_{k+1},x),
  \]
  so $d(x_l,x_k)=d(x_l,x)=d(x_l, x_{k+1})$.
  Therefore,
  \begin{equation*}
    |f(x_{k+1}) - f(x_k)| = d(x_{k+1},x_k).
  \end{equation*}
  Thus $\|f\| \ge 1$.
  Next we will show that $f\in S_{\Lip_0(M)}$.

  Let $x,y \in M$ with $x\neq y$.
  By Lemma~\ref{lem: strongly extreme, first K}, there exists $K\in \mathbb{N}\cup\{0\}$
  such that $x_k=y_k$ and $d(x_{k},x) = d(y_{k},y)$ for $k < K$,
  and
  \begin{equation*}
    \max\{ d(x_{K},x), d(y_{K},y) \} \le d(x,y).
  \end{equation*}
  We get cancellation in the sums defining $f(x)$ and $f(y)$, giving us
  \begin{align*}
    |f(x)-f(y)|
    &=
      \left|\sum_{l=K}^{m_x} (-1)^{l}d(x_l,x)
      - \sum_{l=K}^{m_y} (-1)^{l} d(y_l,y)\right| \\
    &\le
      \max\{d(x_K,x),d(y_K,y)\}\leq d(x,y),
  \end{align*}
  where the first inequality comes from the fact
  that the first sum is in the interval
  $[0,d(x_{K},x)]$ and the second is in
  $[0,d(y_{K},y)]$ (and we may assume that both
  sums are positive by factoring out $(-1)$).

  Now let $g\in \Lip_0(M)$, $\varepsilon > 0$, and assume
  that $\|f \pm g\| \le 1 + \varepsilon$.

  Let $x \neq 0$ and let $(x_k)_{k=0}^{|I_x|}=\mathfrak{s}(x)$.
  Then we can write $g(x)$ as a telescoping sum
  \begin{align*}
    g(x)&=
          \lim_{k \to m_x} g(x_{k})
          =
          \sum_{k=0}^{m_x-1} (g(x_{k+1})-g(x_{k})).
  \end{align*}
  For all $k$, we have
  \begin{align*}
    |f(x_{k+1})-f(x_{k})| + |g(x_{k+1})-g(x_{k})|
    \leq (1+\varepsilon)d(x_{k},x_{k+1}).
  \end{align*}
  Therefore, and since
  $d(x_{k},x_{k+1}) = d(x_{k},x)$,
  \begin{align*}
    |g(x_{k+1})-g(x_{k})|
    &\leq
      (1+\varepsilon)d(x_{k},x_{k+1})-|f(x_{k+1})-f(x_{k})|\\
    &=
      (1+\varepsilon)d(x_{k},x_{k+1})-d(x_{k},x_{k+1})\\
    &=
      \varepsilon d(x_{k},x_{k+1})\\
    &=
      \varepsilon d(x_{k},x).
  \end{align*}
  By Lemma~\ref{lem: x geometric}, we get
  \begin{align*}
    |g(x)|
    &\leq
      \sum_{k=0}^{m_x-1} |g(x_{k+1})-g(x_{k})|
      \leq
      \sum_{k=0}^{m_x} \varepsilon d(x_{k},x)
      \leq
      \frac{\varepsilon}{q(1-q)}d(0,x).
  \end{align*}
  
  If $y=0$, then we have
  \[
  |g(x)-g(y)|=|g(x)|\leq \frac{\varepsilon}{q(1-q)}d(x,y).
  \]
  Assume that $y \in M\setminus\{0,x\}$ and let $(y_k)_{k=0}^{|I_y|}=\mathfrak{s}(y)$.
  Use Lemma~\ref{lem: strongly extreme, second K} to
  find $K\in \mathbb{N}\cup\{0\}$ such that
  $x_{k} = y_{k}$ for $k\leq K$ and
  \begin{equation*}
    \max\{d(x_{K},x), d(y_{K},y)\}
    \le q^{-1}d(x,y),
  \end{equation*}
  By Lemma~\ref{lem: x geometric},
  \begin{align*}
    |\sum_{l=K}^{m_x-1} g(x_{l+1})-g(x_{l})|
    &\le
      \sum_{l=K}^{m_x-1} \varepsilon d(x_{l},x) 
      \le
      \frac{\varepsilon}{q(1-q)} d(x_{K},x)\\
      &<
    \frac{\varepsilon}{q^2(1-q)}d(x,y)
  \end{align*}
  and similarly for $y$.
  Since $x_{k} = y_{k}$ for $k \leq K$, we get
  \begin{align*}
    |g(x)-g(y)|
    &=
      |\sum_{l=K}^{m_x-1}(g(x_{l+1})-g(x_{l}))
      - \sum_{l=K}^{m_y-1}(g(y_{l+1})-g(y_{l}))|\\
    &\leq
      |\sum_{l=K}^{m_x-1}g(x_{l+1})-g(x_{l})|
      +|\sum_{l=K}^{m_y-1}g(y_{l+1})-g(y_{l})|\\
    &<
      \frac{2\varepsilon}{q^2(1-q)}d(x,y).
  \end{align*}
  This shows that if $\|f \pm g_n\| \to 1$
  for $(g_n)$ in $B_{\Lip_0(M)}$,
  then $\|g_n\| \to 0$, i.e. $f$ is a strongly extreme point.
\end{proof}

\providecommand{\bysame}{\leavevmode\hbox to3em{\hrulefill}\thinspace}
\providecommand{\MR}{\relax\ifhmode\unskip\space\fi MR }
% \MRhref is called by the amsart/book/proc definition of \MR.
\providecommand{\MRhref}[2]{%
  \href{http://www.ams.org/mathscinet-getitem?mr=#1}{#2}
}
\providecommand{\href}[2]{#2}

\end{document}